\documentclass[a4paper]{amsart}
\usepackage[utf8]{inputenc}
\usepackage[english]{babel}
\usepackage[foot]{amsaddr}
\usepackage{amsmath, amsthm, amssymb, amsfonts}
\usepackage{bm}
\usepackage{mathrsfs}
\usepackage{mathtools}
\usepackage{xcolor}
\usepackage{bbm}
\usepackage{graphicx}
\usepackage{caption}
\usepackage{subcaption}
\usepackage{todonotes}
\usepackage{hyperref}
\usepackage{nicefrac}
\usepackage{enumitem}
\usepackage{tikz}
\usepackage{microtype}

\hypersetup{colorlinks=true, linkcolor=blue, citecolor=blue}

\theoremstyle{plain}
\newtheorem{theorem}{Theorem}[section]
\newtheorem{corollary}[theorem]{Corollary}
\newtheorem{proposition}[theorem]{Proposition}
\newtheorem{lemma}[theorem]{Lemma}

\theoremstyle{definition}
\newtheorem{definition}[theorem]{Definition}

\theoremstyle{remark}
\newtheorem{remark}[theorem]{Remark}

\numberwithin{equation}{section}

\newcommand{\R}{\mathbb{R}}

\renewcommand{\S}{\mathcal{S}}
\newcommand{\indexset}{\mathcal{I}}
\newcommand{\V}{V}
\newcommand{\set}[1]{\{#1\}}
\newcommand{\norm}[1]{\|#1\|}

\newcommand{\indicator}{\chi}
\newcommand{\lip}{\mathrm{Lip}}

\newcommand{\loc}{\mathrm{loc}}

\newcommand{\eps}{\varepsilon}
\newcommand{\source}{h}
\newcommand{\dx}{\Delta x}

\DeclareMathOperator{\sgn}{sgn}

\DeclareMathOperator*{\esssup}{ess\, sup}

\DeclareMathOperator{\BV}{BV}

\DeclareMathOperator*{\supp}{supp}

\title[A semidiscrete approximation of scalar conservation laws]{Convergence rate for a semidiscrete approximation of scalar conservation laws}

\author{Magnus C. Ørke}
\email{magnusco@math.uio.no}

\date{\today}

\begin{document}

\begin{abstract}
    We propose a semidiscrete scheme for approximation of entropy solutions of one-dimensional scalar conservation laws with nonnegative initial data. The scheme is based on the concept of particle paths for conservation laws and can be interpreted as a finite-particle discretization. A convergence rate of order $\nicefrac{1}{2}$ with respect to initial particle spacing is proved. As a special case, this covers the convergence of the Follow-the-Leader model to the Lighthill--Whitham--Richards model for traffic flow.
\end{abstract}

\maketitle

\section{Introduction}

We consider the one-dimensional, scalar conservation law
\begin{equation} \label{eq:conservation_law}
    \left\{
    \begin{aligned}
        & \partial_t u + \partial_x f(u) = 0 \\
        & u(0) = u_0,
    \end{aligned}
    \right.
\end{equation}
where the initial data is a nonnegative function $u_0 \in \BV \cap L^1(\R)$ and the flux $f$ is assumed to be Lipschitz with $f(0) = 0$ and a well-defined derivative $f'(0)$ at zero.

The present work builds on the equivalence between unique particle paths and entropy solutions for conservation laws established in \cite{fjordholm_maehlen_oerke}. This result is based on interpreting the conservation law~\eqref{eq:conservation_law} as a continuity equation
\begin{equation}\label{eq:cl_as_cont_eq}
    \partial_t u + \partial_x \bigl(a(u) u\bigr) = 0, \quad \text{where} \quad a(u) \coloneqq f(u) / u,
\end{equation}
leading to particle paths
\begin{equation} \label{eq:exact_flow}
    \left\{
        \begin{aligned}
            & \dot{x}_t = \frac{f(u(x_t, t))}{u(x_t, t)}\\
            & x_0 = x.
        \end{aligned}
    \right.
\end{equation}
In particular, it states that if $u$ is the entropy solution of \eqref{eq:conservation_law}, then the ODE \eqref{eq:exact_flow} has a unique solution $t \mapsto x_t$ for all initial conditions $x \in \R$. (Here, $\dot{x} = \frac{d}{dt} x$, and $x_t$ denotes $x$ evaluated at time $t$). Moreover, the collection of solutions forms a continuous forward flow $(x, t) \mapsto X_t(x)$, and the entropy solution satisfies the pushforward formula
\begin{equation} \label{eq:entropy_sol_representation}
    u(t) = (X(t))_\# u_0, \quad \text{i.e.}\quad \int_\R \vartheta(x) u(x, t)\,dx = \int_\R \vartheta(X_t(x)) u_0(x)\,dx
\end{equation}
for all $\vartheta \in C_c(\R)$.

We will formulate a novel finite-particle approximation of entropy solutions of the conservation law \eqref{eq:conservation_law}, centered around discretized versions of the particle path ODE \eqref{eq:exact_flow} and the pushforward representation~\eqref{eq:entropy_sol_representation}.

For concave flux functions $f$, the proposed scheme coincides exactly with the well-known Follow-the-Leader (FtL) model for traffic flow. This connection provides a new theoretical underpinning for the FtL model, and enables us to systematically analyze its rate of convergence to the macroscopic Lighthill--Whitham--Richards (LWR) model. To our knowledge, a convergence proof of this nature for the FtL model has been previously lacking.

A simplified version of the scheme, hereafter referred to as the \textit{particle path scheme}, is as follows. A complete detailed specification will be given in Section~\ref{sec:scheme}.

\begin{enumerate} [label=\textbf{Step \arabic*.}, ref=Step \arabic*]
    \item \label{item:1} \textit{Approximation of initial data}. Initialize a collection of $N$ \textit{particles} $x_0^1 < x_0^2 < \cdots < x_0^N$, and approximate the initial data $u_0$ by a function $v_0$ which is piecewise constant with \textit{local densities}
    \begin{equation*}
        v_0^i = \frac{1}{x_0^{i+1} - x_0^i} \int_{x_0^i}^{x_0^{i+1}} u_0(x)\, dx
    \end{equation*}
    on $(x_0^i, x_0^{i+1})$.
    \vspace{0.05cm}

    \item \label{item:2} \textit{Particle dynamics}. Let each particle $t \mapsto x_t^i$ evolve according to the ODE ${\dot{x}_t^i = V(v_t^{i-1}, v_t^i)}$, where $V$ is a velocity function given by
    \begin{equation} \label{eq:velocity}
        \V(v_l, v_r) \coloneqq
        \left\{
        \begin{aligned}
            & \min_{v\in[v_l, v_r]} a(v) \qquad \text{if}\quad 0 \leq v_l \leq v_r\\
            & \max_{v\in[v_r, v_l]} a(v) \qquad \text{if}\quad 0 \leq v_r \leq v_l.
        \end{aligned}
        \right.
    \end{equation}
    Simultaneously update local densities $t \mapsto v_t^i$ as
    \begin{equation*} \label{eq:local_densities}
        v_t^{i} = \frac{x_0^{i+1} - x_0^i}{x_t^{i+1} - x_t^i} v_0^i.
    \end{equation*}
    Define an approximate solution $v$ from local densities as
    \begin{equation} \label{eq:approx_sol}
        v(x, t) \coloneqq \sum_{i = 0}^{N} v_t^i \indicator_{(x_t^i, x_t^{i+1})}(x),
    \end{equation}
    where $\indicator_{(x^i, x^{i+1})}$ is an indicator function on the interval $(x^i, x^{i+1})$.
    \vspace{0.05cm}

    \item \label{item:3} \textit{Resolution of collisions}. If two or more particles collide, delete the leftmost particles and local densities in the collision, and continue the system. Repeat~\mbox{\ref{item:2}}--\mbox{\ref{item:3}} iteratively up to time $T$.
\end{enumerate}

\subsection{Main results} \label{subsec:main_results}

Rather than analyzing~\mbox{\ref{item:1}}--\mbox{\ref{item:3}}~directly, we show that the function $v$ defined in \eqref{eq:approx_sol} satisfies a continuity equation
\begin{equation} \label{eq:cont_eq}
    \left\{
    \begin{aligned}
        & \partial_t v + \partial_x(A v) = 0 \\
        & v(0) = v_0,
    \end{aligned}
    \right.
\end{equation}
where $A$ is the piecewise linear interpolation of particle velocities with interpolation nodes given by particle positions. That is, $A$ is (up to the first collision, cf.~\eqref{eq:velocity_interpolation_global}) defined as
\begin{equation} \label{eq:velocity_interpolation}
    A(x, t) = \frac{x_t^{i+1}-x}{x_t^{i+1}-x_t^i} V(v_t^{i-1}, v_t^{i}) + \frac{x-x_t^i}{x_t^{i+1}-x_t^i} V(v_t^{i}, v_t^{i+1})
\end{equation}
for $x \in (x_t^i, x_t^{i+1})$. This formulation makes it possible to derive an accurate relationship between the approximation $v$ and the entropy solution of the conservation law~\eqref{eq:conservation_law}. In particular, we prove that $v$ satisfies the entropy-like inequality
\begin{equation} \label{eq:strong_approx_entropy_ineq}
    \partial_t |v - k| + \partial_x \bigl((A v - f(k)) \sgn(v-k)\bigr) \leq 0
\end{equation}
for all nonnegative constants $k \in \R$. As a consequence, we obtain the following theorem.

\begin{theorem}[Main Theorem] \label{thm:main}
    Let $u_0$ be a nonnegative function in $\BV \cap L^1(\R)$ and assume that $f$ is Lipschitz. Then the particle path scheme generates a unique approximation $v$ which satisfies
    \begin{equation} \label{eq:main_stability_result}
        \norm{v(T) - u(T)}_{L^1(\R)} \leq \norm{v_0 - u_0}_{L^1(\R)} + 2 \sqrt{2 |u_0|_{\BV(\R)} \norm{A v - f(v)}_{L^1(\R \times (0, T))}},
    \end{equation}
    where $u$ is the entropy solution of the conservation law \eqref{eq:conservation_law}.
\end{theorem}

One should note that the function $A$ in the stability estimate \eqref{eq:main_stability_result} is the global-in-time version of~\eqref{eq:velocity_interpolation}, defined in \eqref{eq:velocity_interpolation_global} (see also Section~\ref{subsec:collision_resolution}). It follows the same principle but is more involved to write down due to particle collisions.

Theorem~\ref{thm:main} is a general result, imposing no assumptions on the distribution of initial particles. As the following corollary exemplifies, explicit convergence rates can be derived under additional conditions.

\begin{corollary} \label{cor:main}
    In addition to the assumptions of Theorem~\ref{thm:main}, assume that $f'$ is Lipschitz and that $u_0$ has compact support. Further, assume that $\supp(u_0) \subset [x_0^1, x_0^N]$ and $x_0^{i+1} - x_0^{i} \leq \dx^*$ for all $i \in \set{1, \dots, N-1}$. Then the approximation error is bounded by
    \begin{equation*}
        \norm{v(T) - u(T)}_{L^1(\R)} \leq |u_0|_{\BV(\R)} \biggl(\dx^* + 2\sqrt{T [f']_\lip \norm{u_0}_{L^\infty(\R)} \dx^*}\biggr).
    \end{equation*}
    That is, the scheme is of order $\nicefrac{1}{2}$ with respect to initial distance between particles.
\end{corollary}

\subsection{Illustrative example} \label{subsec:example}

To illustrate the nature of the approximation, we consider the following problem for Burgers' equation:
\begin{equation*} \label{eq:example_prob}
    \partial_t u + \partial_x \biggl(\frac{1}{2}u^2\biggr) = 0, \qquad u_0(x) = 
    \begin{cases}
        3 & \text{for } 0 < x < 1 \\
        1 & \text{else}.
    \end{cases}
\end{equation*}
The unique entropy solution is given by
\begin{equation} \label{eq:entropy_sol_example}
    u(x, t) = 
    \begin{cases}
        \frac{x}{t} &  \text{if } t < x < 3t \\
        3 & \text{if } 3t < x < 1 + 2t \\
        1 & \text{else},
    \end{cases}
\end{equation}
which features both a rarefaction region and a shock. Figure~\ref{fig:approximation_example} compares the approximation~$v$ generated by the particle path scheme with the entropy solution for two different spacings of initial particles. The underlying idea of the approximation becomes evident in the particle plots to the right: Particles $(x_t^i)_{i = 1}^N$ divide the mass of $u$ into $N-1$ parts $(v_t^i)_{i = 1}^{N-1}$. The mass between adjacent particles remains constant over time and is continually reaveraged, making $v$ a piecewise constant function. See \cite[Example 6.1]{fjordholm_maehlen_oerke} for the exact particle paths of the entropy solution in this example.

\begin{figure}
    \centering
    \begin{subfigure}[t]{\textwidth}
        \centering
        \includegraphics[width=0.45\textwidth]{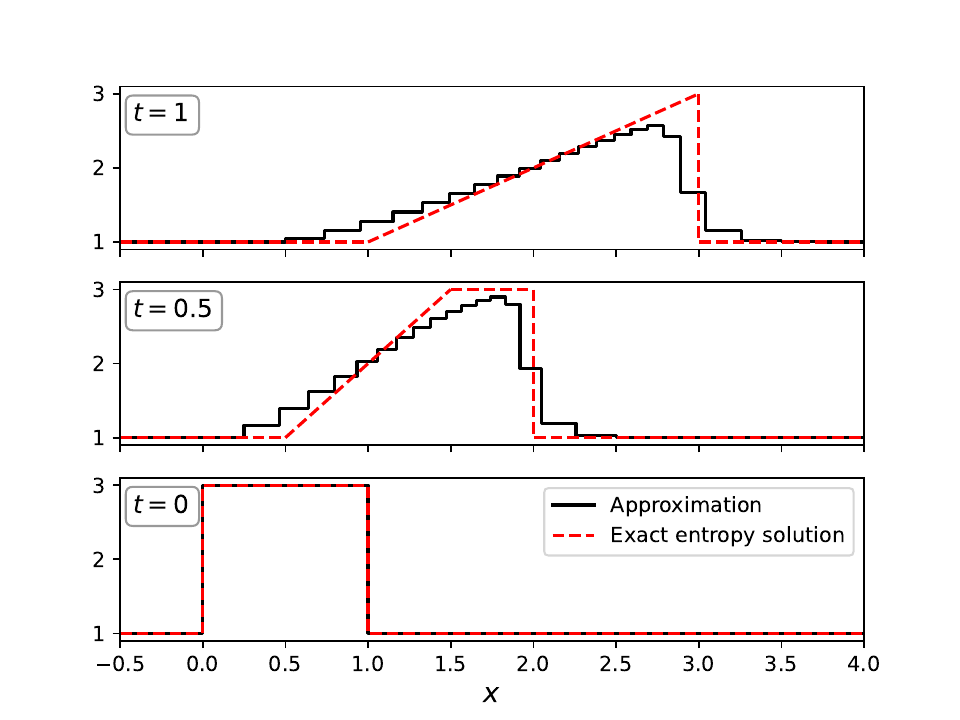}
        \includegraphics[width=0.45\textwidth]{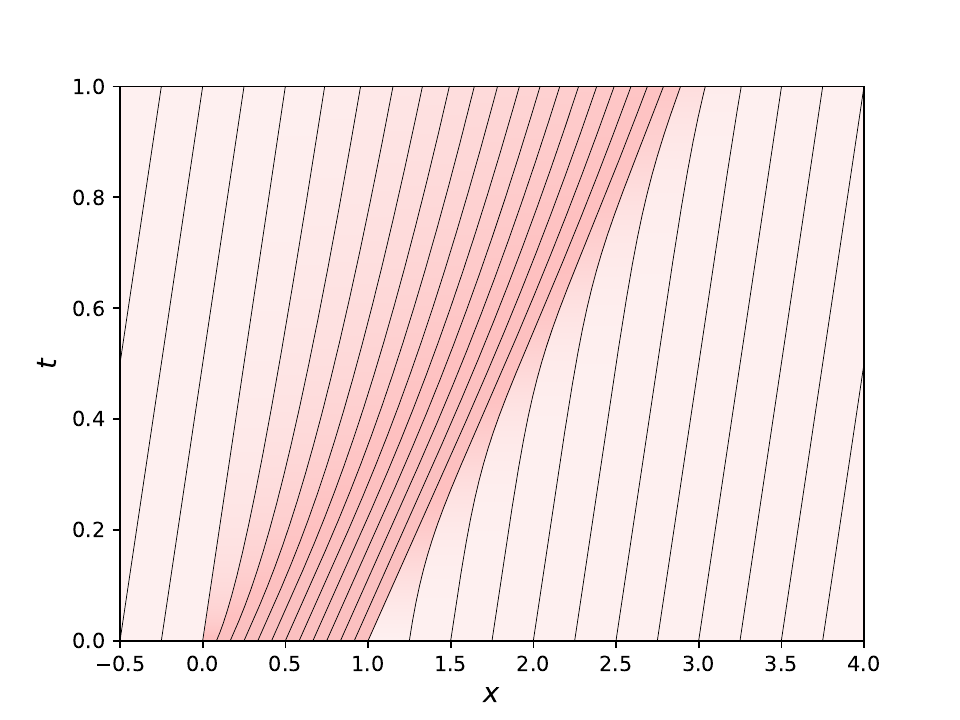}
        \caption{Coarse approximation.} \label{fig:coarse_approximation}
    \end{subfigure}

    \begin{subfigure}[t]{\textwidth}
        \centering
        \includegraphics[width=0.45\textwidth]{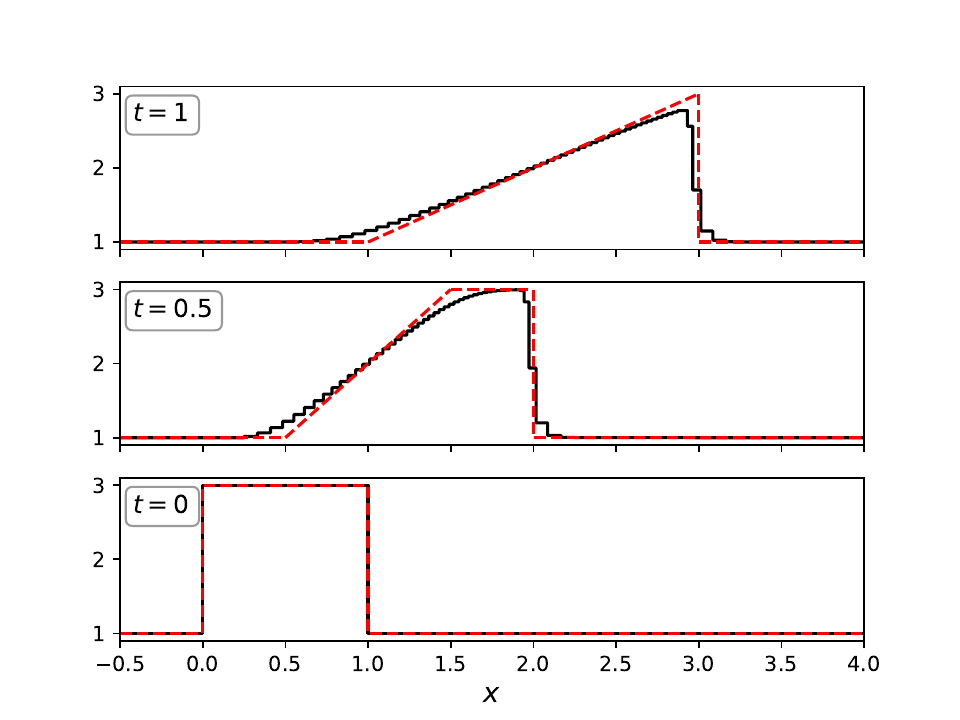}
        \includegraphics[width=0.45\textwidth]{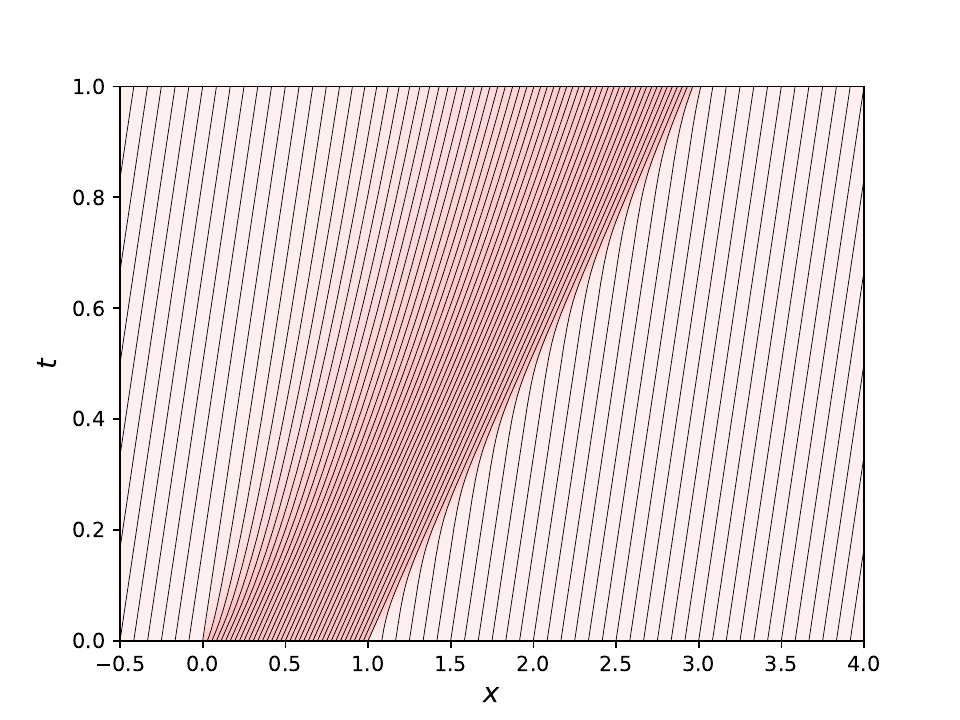}
        \caption{Finer approximation.} \label{fig:fine_approximation}
    \end{subfigure}
    \caption{Approximations of \eqref{eq:entropy_sol_example} using different numbers of initial particles. The particle paths were computed with the forward Euler method using a small time step. In the particle path plots (right), the background shading represents the magnitude of the approximate solution.}
    \label{fig:approximation_example}
\end{figure}

\subsection{Comparison to front tracking} \label{subsec:front_tracking}

The difference between our approach and methods like front tracking can be understood through the interpretation of the conservation law \eqref{eq:conservation_law}. Standard methods often rely on the quasi-linear form
\begin{equation*}
    \partial_t u + f'(u) \partial_x u = 0.
\end{equation*}
In this view, approximating the characteristic speed $f'(u)$ by a piecewise constant function (which corresponds to approximating the flux $f(u)$ by a piecewise linear function) effectively assigns a specific velocity to the values of the state variable $u$. The dynamics are then governed by rules for how these constant states interact, such as rarefaction waves and the Rankine--Hugoniot jump conditions which are used in front tracking to construct approximate entropy solutions.

In contrast, the particle path scheme stems from interpreting \eqref{eq:conservation_law} as a continuity equation \eqref{eq:cl_as_cont_eq} with a density-dependent velocity field $a(u) = f(u)/u$. We then approximate the velocity field $a(u(x,t))$, viewed as a function of the spatial variable $x$, by a piecewise constant function. This assigns a specific velocity, and thus a rate of compression or expansion, to spatial intervals between particles. A fundamental difference between our approach and front tracking is therefore that front tracking approximates the flux $f$ as a function of the state variable $u$, while our method approximates the velocity field $a(u(x,t))$ as a function of the spatial variable $x$. 

Since $u$ is generally discontinuous, the velocity field $a(u)$ will also have discontinuities. This necessitates a careful choice for the velocity at these spatial jumps. As we will demonstrate, the specific choice leading to the entropy solution is the velocity defined in \eqref{eq:velocity}.

\subsection{Application to traffic modelling} \label{subsec:applications}

In the particular case of concave flux $f$, the proposed scheme coincides with the Follow-the-Leader (FtL) (see e.g. \cite{holden_risebro_2018}) model for traffic flow and establishes its convergence rate to the Lighthill--Whitham--Richards (LWR) \cite{lighthill_whitham_1955,richards_1956} model. This convergence has been studied by several others \cite{aw_klar_rascle_materne, colombo_rossi, francesco_rosini_2015, goatin_rossi_2017, holden_risebro_2018_short, holden_risebro_2018}, but to our knowledge, an explicit convergence rate like \eqref{eq:main_stability_result} has not previously been available.

The LWR model describes average vehicular density $v$ on a macroscopic level through the conservation law
\begin{equation*}
    \partial_t u + \partial_x(a(u) u) = 0,
\end{equation*}
where $a$ is a nonincreasing Lipschitz velocity field (in the literature, the density is usually denoted by $\rho$ and the velocity field $v$ so that the equation reads $\partial_t \rho + \partial_x (v(\rho)\rho) = 0$, but we have kept the notation used in this paper for consistency). On the other hand, the FtL model is a microscopic model for individual vehicles $t \mapsto z_t^i$, each of which is assigned a velocity
\begin{equation} \label{eq:ftl_particle}
    \dot{z}_t^i = a\bigl(v_t^i\bigr), \quad \text{where}\quad v_t^i \coloneqq \frac{l}{z_t^{i+1} - z^i_t}
\end{equation}
and $l > 0$ denotes the length of the vehicles. Since $a$ is assumed to be nonincreasing, the FtL model is a special case of our scheme: If $v_t^{i-1} \leq v_t^{i}$, then
\begin{equation*}
    V(v_t^{i-1}, v_t^{i}) = \min_{v \in [v_t^{i-1}, v_t^i]} a(v) = a(v_t^{i})
\end{equation*}
and if on the other hand $v_t^{i-1} \geq v_t^{i}$, then
\begin{equation*}
    V(v_t^{i-1}, v_t^{i}) = \max_{v \in [v_t^{i}, v_t^{i-1}]} a(v) = a(v_t^{i}).
\end{equation*}
In both cases, the velocity $V$ defined in \eqref{eq:velocity} coincides with the velocity for the FtL model in \eqref{eq:ftl_particle}.

\subsection{Outline of the paper}

Section~\ref{sec:preliminaries} covers theory on conservation laws, particle paths, and continuity equations. Section~\ref{sec:scheme} defines the particle path scheme in full detail, and Section~\ref{sec:existence_uniqueness} establishes existence and uniqueness of solutions to the system of ODEs in~\mbox{\ref{item:2}}. The core of the paper is Section~\ref{sec:pde_formulation}, wherein justification of the continuity equation \eqref{eq:cont_eq} and proof of the approximate entropy inequality \eqref{eq:strong_approx_entropy_ineq} is given. Theorem \ref{thm:main} is proved in Section~\ref{sec:convergence_rate}. In Section~\ref{sec:discussion} we mention possible future work.

\section{Preliminaries} \label{sec:preliminaries}

We begin by reviewing key concepts and results needed for the paper.

\subsection{Scalar conservation laws} \label{subsec:cl}

A solution of the conservation law \eqref{eq:conservation_law} means a weak solution, that is, a function $u \in L^\infty(\R \times (0, T))$ which satisfies
\begin{equation*} \label{eq:weak_cl}
    \int_0^T \int_\R u \partial_t \varphi + f(u) \partial_x \varphi\, dx\,dt + \int_\R u_0(x) \varphi(x, 0)\, dx = 0
\end{equation*}
for all $\varphi \in C^\infty_c(\R \times [0, T))$. An entropy solution is a solution which satisfies the entropy inequality
\begin{equation} \label{eq:weak_general_entropy}
    \int_0^T \int_\R \eta(u) \partial_t \varphi + q(u) \partial_x \varphi\, dx\,dt + \int_\R \eta(u_0(x)) \varphi(x, 0)\, dx \geq 0
\end{equation}
for all nonnegative $\varphi \in C^\infty_c(\R \times [0, T))$ and all entropy pairs $(\eta, q)$ (where ${\eta\colon \R \to \R}$ is a convex function and $q \colon \R \to \R$ is such that $q' = \eta' f'$). Kruzkhov \cite{kruzkov_1970} proved that there is a unique entropy solution of \eqref{eq:conservation_law} for any ${u_0\in L^\infty(\R)}$. It is enough that \eqref{eq:weak_general_entropy} holds for entropy pairs
\begin{equation*}\label{eq:kruzkhov_entropies}
    \eta_k(u) = |u-k|, \qquad q_k(u) = \sgn(u-k)(f(u)-f(k))
\end{equation*}
for all $k \in \R$. (Here and throughout, $\sgn$ denotes the signum function. It returns the sign of its argument or $0$ if the argument is zero.) If $u_0$ is an absolutely integrable function of bounded variation, then so is the entropy solution for all positive times.

\subsection{Particle paths for conservation laws} \label{subsec:particle_paths}

The concept of particle paths for the conservation law \eqref{eq:conservation_law} comes from its interpretation as a continuity equation \eqref{eq:cl_as_cont_eq}. The following theorem describes the link between entropy solutions and particle paths.

\begin{theorem}[Fjordholm, Mæhlen, Ørke \cite{fjordholm_maehlen_oerke}]\label{thm:particle_paths}
    Let $f\in C^1(\R)$ and $u_0\in \BV_{\loc}\cap L^\infty(\R)$. If ${u\in L^\infty(\R\times\R_+)}$ is a weak solution of \eqref{eq:conservation_law} with $u(t)\in \BV_{\loc}(\R)$ for a.e.~$t\geq 0$, then the following are equivalent:
    \begin{enumerate}[label=(\roman*)]
    \item $u$ is the entropy solution of \eqref{eq:conservation_law}.
    \item The ODE
    \begin{equation} \label{eq:flow_with_k}
        \left\{
        \begin{aligned}
            & \dot{x}_t = \frac{f(u(x_t,t)) - f(k)}{u(x_t,t) - k} \qquad \text{for}\ t > s\\
            & x_s = x
        \end{aligned}
        \right.
    \end{equation}
    is well-posed in the Filippov sense for all $x \in \R$, $s \geq 0$ and all $k \in \R$.
    \end{enumerate}
    Moreover, for any $k \in \R$, the entropy solution satisfies
    \begin{equation*}\label{eq:representation-formula}
        u(t) = k + (X^k_t)_\#(u_0 - k) \qquad \text{for } t\geq0
    \end{equation*}
    where $X^k=X^k_t(x)$ is the flow of the ODE \eqref{eq:flow_with_k}.
\end{theorem}

The ODE \eqref{eq:flow_with_k} (and \eqref{eq:exact_flow}) is interpreted in the Filippov sense (see \cite{Filippov60}), as weak solutions of \eqref{eq:conservation_law} are generally discontinuous. In particular, when a particle encounters a shock, its velocity must typically be assigned from an infinite set of possible values. The prototypical example of this situation is the Riemann problem: Let $u$ be the entropy solution of \eqref{eq:conservation_law} with initial data
\begin{equation*}
    u_0(x) =
    \begin{cases}
        u_l & \text{for }x<0\\
        u_r & \text{for } x>0
    \end{cases}
\end{equation*}
for $u_l, u_r\in\R$. As shown in \cite[Theorem~1.6]{fjordholm_maehlen_oerke}, the unique (Fillipov) solution of \eqref{eq:exact_flow} starting at $x = 0$ is the straight line $x_t=Vt$, with velocity $V = V(u_l, u_r)$ given by
\begin{equation}\label{eq:velocityForRiemannParticle}
V(u_l, u_r) =
    \begin{cases}
        \min_{v\in[u_l, u_r]}a(v) & \text{if}\ 0<u_l\leq u_r\vspace{3pt}\\
        \max_{v\in[u_r, u_l]}a(v) & \text{if}\ 0<u_r\leq u_l\vspace{3pt}\\
        (f_\smallsmile)'(0) & \text{if}\ u_l\leq 0\leq u_r\vspace{3pt}\\
        (f_\smallfrown)'(0) & \text{if}\ u_r\leq 0\leq u_l\vspace{3pt}\\
        \max_{v\in[u_l, u_r]} a(v) & \text{if}\ u_l\leq u_r<0\vspace{3pt}\\
        \min_{v\in[u_r, u_l]} a(v) & \text{if}\ u_r\leq u_l<0,
    \end{cases}
\end{equation}
where $f_\smallsmile$ and $f_\smallfrown$ are the convex and concave envelopes of $f$ between $u_l$ and $u_r$ (definitions can be found in \mbox{\cite[Chapter 2]{holden_risebro_2015}}). Comparing with \eqref{eq:velocity}, it is clear that our choice of particle velocity in~\mbox{\ref{item:2}} is simply a special case of \eqref{eq:velocityForRiemannParticle} for nonnegative left and right states $u_l$ and $u_r$.

\subsection{Continuity equations} \label{subsec:continuity_eq}

Consider the continuity equation
\begin{equation} \label{eq:general_continuity} 
    \partial_t u + \partial_x(b u) = \source
\end{equation}
with initial data $u_0 \in L^\infty(\R)$, velocity field $b \in L^\infty(\R \times (0, T))$, and a source $\source \in L^\infty((0, T); L^1(\R))$. A weak solution of \eqref{eq:general_continuity} is a function $u \in L^\infty(\R \times (0, T))$ which satisfies
\begin{equation} \label{eq:weak_general_continuity}
    \int_0^T \int_{\R} \bigl(\partial_t \varphi + b \partial_x \varphi\bigr) u\, dx\,dt + \int_0^T \int_{\R} \varphi h\, dx\,dt + \int_{\R} u_0(x) \varphi(x, 0)\, dx = 0
\end{equation}
for all $\varphi \in C^\infty_c(\R \times [0, T))$.

\begin{remark}
    While measure-valued solutions of \eqref{eq:general_continuity} are generally more natural in this setting (see e.g.~\cite{ambrosio_gigli_savare_2005}), we include only what is presently needed.
\end{remark}

The continuity equation can be solved via the system of ODEs
\begin{equation} \label{eq:one_sided_lipschitz_flow}
    \left\{
    \begin{aligned}
        & \dot{x}_t = b(x_t, t) \qquad \text{for a.e.~}\ t > s\\
        & x_s = x.
    \end{aligned}
    \right.
\end{equation}
Assume that $b$ is continuous in space for a.e.~$t$, so that \eqref{eq:one_sided_lipschitz_flow} can be understood in the usual sense (in particular, we will not need the Filippov solution concept). The following lemma is a standard extension of the Cauchy--Lipschitz theory for ODEs; see e.g.~\cite{conway_1967,lions_seeger_2024}.

\begin{lemma}[Forward flow] \label{lemma:forward_flow}
    Assume that $x \mapsto b(x, t)$ is continuous and one-sided Lipschitz, i.e.~there is a constant $C > 0$ such that
    \begin{equation} \label{eq:one_sided_lip}
        \bigl(b(x, t) - b(y, t)\bigr)(x-y) \leq C (x-y)^2
    \end{equation}
    for all $x, y \in \R$, for a.e.~$t \in (0, T)$. Then the system \eqref{eq:one_sided_lipschitz_flow} generates a unique forward Lipschitz flow $X = X_t(x, s)$:
    \begin{enumerate}[label=(\roman*)]
        \item For all $(x, s) \in \R \times [0, T]$, the ODE \eqref{eq:one_sided_lipschitz_flow} has a unique continuously differentiable solution $t \mapsto x_t$ on $(s, T)$.
        \item The function $(x, s, t) \mapsto X_t(x, s)$ is jointly Lipschitz for $0 < s \leq t < T$ and $x \in \R$.
        \item The identity $X_{s, s}(x) = x$ holds for all $(x, s) \in \R \times [0, T]$. Moreover, 
        \begin{equation*}
            X_t(X_r(s, x), r) = X_t(x, s)
        \end{equation*}
        for all $0 \leq s \leq r \leq t \leq T$ and $x \in \R$.
    \end{enumerate}
\end{lemma}

We denote $X_t(x, 0)$ by $X_t(x)$ for simplicity.

\begin{lemma}[Representation formula] \label{lemma:general_representation_formula}
    Let $b$ satisfy the assumptions of Lemma~\ref{lemma:forward_flow} and assume that $u \in L^\infty(\R \times (0, T))$ is a weak solution of \eqref{eq:general_continuity}. Then $u$ satisfies
    \begin{equation*}\label{eq:continuity_representation}
        u(t) = (X_t)_\# u_0 + \int_0^t (X_t(r))_\# \source(r)\, dr,
    \end{equation*}
    in the sense of
    \begin{equation} \label{eq:push_forward}
        \int_\R \vartheta(x) u(x, t) \,dx = \int_\R \vartheta(X_t(x)) u_0(x) \,dx + \int_0^t \int_\R \vartheta(X_t(x, r)) h(x, r) \,dx\,dr
    \end{equation}
    for all $t \in [0, T]$ and $\vartheta \in C_c(\R)$, where $X$ is the unique forward Lipschitz flow generated by \eqref{eq:one_sided_lipschitz_flow}.
\end{lemma}

This is a special case of \cite[Theorem~6.1, Remark~6.2]{fjordholm_maehlen_oerke} (see also \cite{poupaud_rascle_1997}). Though the proof is much simpler in this case, we omit it.

\section{The scheme} \label{sec:scheme}

In this section, we provide further details on the particle path scheme introduced in \mbox{\ref{item:1}}--\mbox{\ref{item:3}} and fix notation which will be used throughout.

\subsection{Approximation of initial data} \label{subsec:approx_initial_data}

Let $(x_0^i)_{i=1}^N$ be a strictly increasing tuple of real values representing the initial positions of $N$ particles. Extend the tuple symbolically by ${x_0^{0} \coloneqq -\infty}$ and ${x_0^{N+1} \coloneqq \infty}$. Recalling that $u_0$ is assumed to be nonnegative, we define
\begin{equation*} \label{eq:intial_approx_pieces}
    v_0^i \coloneqq \frac{1}{\dx_0^{i}} \int_{x_0^i}^{x_0^{i+1}} u_0(x)\, dx
\end{equation*}
for $i \in \set{0,\dots, N}$, where $\dx^i_0 \coloneqq x_0^{i+1} - x_0^i$. Note in particular that $v_0^0 = 0$ and $v_0^N = 0$. Let $\indicator_I$ denote the indicator function on a set $I \subset \R$. Then
\begin{equation*} \label{eq:initial_approximation}
    v_0(x) \coloneqq \sum_{i = 0}^{N} v_0^i \indicator_{(x_0^i, x_0^{i+1})}(x)
\end{equation*}
is a piecewise constant approximation of $u_0$.

\subsection{Particle dynamics} \label{subsec:dynamics}

Fix $x_t^{0} = -\infty$ and $x_t^{N+1} = \infty$, along with $v^0_t = 0$ and $v_t^N = 0$, for all $t \in [0, T]$. Let particle positions $(x_t^i)_{i = 1}^{N}$ and local densities $(v_t^i)_{i = 1}^{N-1}$ evolve according to the system
\begin{equation}
    \left\{
    \begin{aligned} \label{eq:particle_ode}
        & \dot{x}_t^i = V\bigl(v_t^{i-1}, v_t^{i}\bigr) \qquad \text{for}\ 0 < t < t_1 \\
        & x^{i}|_{t = 0} = x_{0}^{i},
    \end{aligned}
    \right.
\end{equation}
coupled with
\begin{equation} \label{eq:value_ode}
    v_t^{i} = \frac{\dx_0^{i}}{\dx_t^{i}} v_0^i,
\end{equation}
where $t_1$ is the first \textit{collision} time of two or more particles, defined by
\begin{equation} \label{eq:collision_time}
    t_1 = \sup\bigl\{t \in (0, T)\colon x_t^{i} < x_t^{i+1}\ \text{for all}\ i \in \set{1, \dots, N-1}\bigr\} \wedge T,
\end{equation}
and $V$ is the velocity field from \eqref{eq:velocity}. (Here, $x \wedge y$ denotes $\min\{x, y\}$) For brevity, we shall sometimes write $(x_t^i)_{i = 1}^{N}$ and $(v_t^i)_{i = 1}^{N-1}$ as $(x_t^i)$ and $(v_t^i)$.

An approximate solution of \eqref{eq:conservation_law} can now be defined as
\begin{equation} \label{eq:approx_sol_local}
    v(x, t) = \sum_{i = 0}^{N} v_t^i \indicator_{(x_t^i, x_t^{i+1})}(x)
\end{equation}
on $\R \times (0, t_1)$.

\subsection{Resolution of collisions} \label{subsec:collision_resolution}

When particles collide, the system \eqref{eq:particle_ode}--\eqref{eq:value_ode} is stopped. The leftmost particles and corresponding local densities are deleted, and then the system is continued (see Fig.~\ref{fig:collision} for an illustration of this situation).

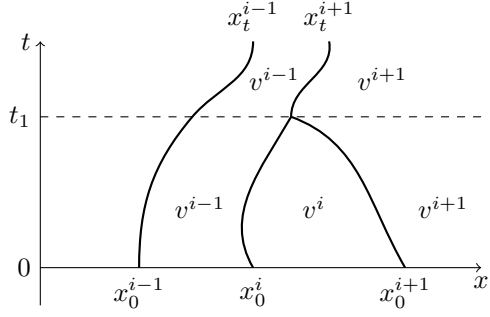
\begin{figure}
    \centering    
    \begin{tikzpicture}
        \draw[->] (-2.8,0) node[left] {$0$} -- (3,0) node[below] {$x$};
        \draw[->] (-2.8, -0.5) -- (-2.8, 3) node[left] {$t$};
    
        \draw[dashed] (-2.8,2) node[left] {$t_1$} -- (3,2);

        \draw[thick, smooth] (-1.5,0) node[below] {$x^{i-1}_0$} to[out=90, in=230] (-0.8,2) to[out=50, in=270] (0,3) node[above] {$x^{i-1}_t$};

        \draw[thick, smooth, tension=1] (-0, 0) node[below] {$x^i_0$} to[out=120, in=240] (0.5, 2) to[out=90, in=280] (1,3) node[above] {$x^{i+1}_t$};
        
        \draw[thick, smooth, tension=1] (2,0) node[below] {$x^{i+1}_0$} to[out=120, in=-20] (0.5,2);
    
        \node at (-0.7, 0.8) {$v^{i-1}$};
        \node at (0.8, 0.8) {$v^i$};
        \node at (2.5, 0.8) {$v^{i+1}$};
        \node at (0.28, 2.5) {$v^{i-1}$};
        \node at (1.7, 2.5) {$v^{i+1}$};
    \end{tikzpicture}

    \caption{Particles $x^i$ and $x^{i+1}$ collide at time $t_1$. The particle $x^{i}$ and the local density $v^i$ are deleted before the system is restarted.}
    \label{fig:collision}
\end{figure}

More specifically, let $\indexset_1$ be the index set of remaining particles after $t_1$. That is,
\begin{equation*}
    \indexset_1 \coloneqq \bigl\{i \in \set{1, \dots, N}\colon x_{t_1}^{i} \neq x_{t_1}^{i+1} \bigr\}.
\end{equation*}
Remaining particles and local densities can thus be written as $(x^{i})_{i \in \indexset_1}$ and $(v^{i})_{i \in \indexset_1}$. Let $N_1 \coloneqq |I_1|$, and define $\pi_1$ to be the increasing, one-to-one function which maps the index set $\set{1, \dots, N_1}$ to $\indexset_1$. Then remaining particles can equivalently be written as $\bigl(x^{\pi_1(i)}\bigr)_{i = 1}^{N_1}$.

After the collision at $t_1$, the system \eqref{eq:particle_ode}--\eqref{eq:value_ode} is continued with updated indices $\indexset_1$ and initial conditions $(x_{t_1}^{\pi_1(i)})_{i = 1}^{N_1}$ and $(x_{t_1}^{\pi_1(i)})_{i = 1}^{N_1-1}$.

Define collision times recursively as
\begin{equation*}
    t_{j+1} \coloneqq \sup\bigl\{t \in (t_j, T)\colon x_t^{\pi_j(i)} < x_t^{\pi_j(i+1)}\ \text{for all}\ i \in \set{1, \dots, N_j}\bigr\}.
\end{equation*}
The set of all collision times $(t_j)_{j \geq 0}$ will be denoted by $\mathcal{T}_c$. The index set of remaining particles after $t_{j+1}$ is then
\begin{equation*}
    \indexset_{j+1} \coloneqq \bigl\{\pi_j(i) \colon x_{t_{j+1}}^{\pi_j(i)} \neq x_{t_{j+1}}^{\pi_j(i+1)},\ i \in \set{1, \dots, N_j} \bigr\}.
\end{equation*}
Let $N_{j+1} \coloneqq |\indexset_{j+1}|$ and let $\pi_{j+1}$ be the increasing, one-to-one function mapping $\set{1, \dots, N_{j+1}}$ to $\indexset_{j+1}$. Set $t_0 = 0$ and let $\pi_0(i) = i$ for all $i \in \set{0, \dots, N}$. In view of the above definitions, we denote by $(x^\pi)$ and $(v^\pi)$ the global-in-time particles and local densities, given by
\begin{equation}
    (x_t^\pi) = (x_t^{\pi(i)})_{i=1}^{N_j}, \qquad (v_t^\pi)= (v_t^{\pi(i)})_{i=1}^{N_j-1}
\end{equation}
for $t \in (t_j, t_{j+1})$.

This allows us to build a global approximate solution of \eqref{eq:conservation_law}, defined piecewise in time as
\begin{equation} \label{eq:approx_sol_global}
    v(x, t) \coloneqq \sum_{i = 0}^{N_j+1} v_t^{\pi_j(i)} \indicator_{\bigl(x_t^{\pi_j(i)}, x_t^{\pi_j(i+1)}\bigr)}(x)
\end{equation}
for $(x, t) \in \R \times (t_j, t_{j+1})$.

\section{Existence and uniqueness of the approximation} \label{sec:existence_uniqueness}

In this section we consider the system of ODEs \eqref{eq:particle_ode}--\eqref{eq:value_ode} which is at the core of the particle path scheme. We establish first some a priori properties of solutions, assuming their existence. We then use these a priori properties and classical Cauchy--Lipschitz theory to show that solutions indeed exist and are unique.

The integer $N$ will be fixed throughout the section. We assume that initial data $(v_0^i)_{i = 1}^{N-1}$, $(x_0^i)_{i = 1}^N$, as constructed in \mbox{\ref{item:1}} from nonnegative $u_0 \in \BV \cap L^1(\R)$, is given. Since there is only $N$ particles there can be at most $N-1$ collisions (note that new particles are never created), so it will suffice to analyze the scheme locally in time. Thus, we consider only the interval $(0, t_1)$.

\subsection{A priori properties}

We begin by defining a suitable notion of solutions of \eqref{eq:particle_ode}--\eqref{eq:value_ode}.

\begin{definition} \label{def:system_solution}
    A solution of the system \eqref{eq:particle_ode}--\eqref{eq:value_ode} on an interval $(0, t')$ is a pair $(x^{i})_{i = 1}^N$, $(v^{i})_{i = 1}^{N-1}$ such that
    \begin{enumerate}[label=(\roman*)]
        \item $(x^{i})_{i = 1}^N \colon (0, t') \to \R^N$ is continuously differentiable and satisfies~\eqref{eq:particle_ode},
        \item $x_t^1 < x_t^2 < \cdots < x_t^N$ for all $t \in (0, t')$,
        \item $(v^{i})_{i = 1}^{N-1}\colon (0, t') \to \R^{N-1}$ satisfies~\eqref{eq:value_ode} for all $t \in (0, t')$.
    \end{enumerate}
\end{definition}

The following proposition gathers essential characteristics of such solutions. To simplify notation, let $v_0^* \coloneqq \max_{i \in \set{0, \dots, N}} v_0^i$ and moreover
\begin{equation*} \label{eq:a_bounds}
    a_{\min} \coloneqq \min_{v \in [0, v_0^*]} a(v), \qquad a_{\max} \coloneqq \max_{v \in [0, v_0^*]} a(v),
\end{equation*}
where we recall that $a(v) = f(v)/v$.

\begin{proposition} \label{prop:a_priori_properties}
    Assume that $(x^{i})_{i = 1}^N$, $(v^{i})_{i = 1}^{N-1}$ is a solution of \eqref{eq:particle_ode}--\eqref{eq:value_ode} on a given interval $(0, t')$ according to Definition~\ref{def:system_solution}. Then, for all $t \in (0, t')$:
    \begin{enumerate}[label=(\roman*)]
        \item (Mass conservation) The mass between particles is conserved, i.e.
        \begin{equation*} \label{eq:mass_conservation}
            \dx_t^i v_t^i = \dx_0^i v_0^i
        \end{equation*}
        for all $i \in \set{0, \dots, N}$.
        \item (Maximum principle) The local densities satisfy
        \begin{equation} \label{eq:value_bounds}
            \frac{\dx_0^i}{\dx_0^{i} + t (a_{\max} - a_{\min})} v_0^i \leq v_t^i \leq v_0^*
        \end{equation}
        for all $i \in \{1, \dots, N - 1\}$.
        \item (Particle separation) The distances between particles satisfy
        \begin{equation} \label{eq:particle_bounds}
            \frac{v_0^i}{v_0^*} \dx_0^{i} \leq \dx_t^{i} \leq \dx_0^{i} + t (a_{\max} - a_{\min})
        \end{equation}
        for all $i \in \{1, \dots, N - 1\}$.
        \item (Diminishing total variation) The total variation of the approximate solution~\eqref{eq:approx_sol_local} is nonincreasing, i.e.
        \begin{equation} \label{eq:tv_dinimishing}
            |v(t)|_{\BV(\R)} \leq |v_0|_{\BV(\R)},
        \end{equation}
        where $|\cdot|_{\BV(\R)}$ denotes the total variation on $\R$.
    \end{enumerate}
\end{proposition}

\begin{proof}
    (i) The conservation of mass follows directly from \eqref{eq:value_ode}.

    (ii) We first claim that the local densities $(v^{i})$ satisfy
    \begin{equation} \label{eq:max_principle}
        \left\{
        \begin{aligned}
            & \dot{v}_t^i \leq 0 \qquad \text{if} \quad v_t^{i-1} \leq v_t^i \geq v_t^{i+1} \\
            & \dot{v}_t^i \geq 0 \qquad \text{if} \quad v_t^{i-1} \geq v_t^i \leq v_t^{i+1}
        \end{aligned}
        \right.
    \end{equation}
    for all $t \in (0, t')$ and $i \in \{1, \dots, N - 1\}$. Indeed, if $v^i_t = 0$ for some $t \in (0, t')$, then $v_t^i \equiv 0$ on $[0, t')$ due to \eqref{eq:value_ode}, and \eqref{eq:max_principle} trivially holds. Let therefore $v_t^i$ be strictly positive and assume that $v_t^{i-1} \leq v_t^i \geq v_t^{i+1}$ for some $t \in (0, t')$. Then
    \begin{equation*}
        \dot{x}_t^i = V(v_t^{i-1}, v_t^i) = \min_{v \in [v_t^{i-1}, v_t^i]} a(v) \leq a(v_t^i) \leq \max_{v \in [v_t^{i+1}, v_t^i]} a(v) = V\bigl(v_t^i, v_t^{i+1}\bigr) = \dot{x}_t^{i+1}
    \end{equation*}
    which implies that
    \begin{equation*} \label{eq:value_derivative_bound}
        \dot{v}_t^i = \frac{d}{dt} \frac{x_0^{i+1} - x_0^{i}}{x_{t}^{i+1} - x_{t}^{i}} v_0^{i} = - v_t^i \frac{\dot{x}_t^{i+1} - \dot{x}_t^i}{x_{t}^{i+1} - x_{t}^{i}} \leq 0
    \end{equation*}
    whenever $x_t^{i+1} > x_t^i$. The second inequality in \eqref{eq:max_principle} follows by a similar argument.

    It is now straightforward to infer from \eqref{eq:max_principle} that $0 \leq v^i_t \leq v_0^*$ for all $t \in (0, t')$. Since $x^i_t$ satisfies \eqref{eq:particle_ode}, we have moreover that
    \begin{equation*}
        x^i_0 + t a_{\min} \leq x^i_t \leq x^i_0 + t a_{\max}
    \end{equation*}
    for all $t \in (0, t')$ and $i \in \set{1, \dots, N}$, where we have used ${a_{\min} \leq V(v_t^{i-1}, v_t^i) \leq a_{\max}}$ on $(0, t')$. This yields $\dx_t^i \leq \dx_0^i + t(a_{\max} - a_{\min})$, which inserted in \eqref{eq:value_ode} proves~\eqref{eq:value_bounds}.

    (iii) Since $\dx^{i}_t = \frac{v^i_0}{v^i_t}\dx^{i}_0$, the inequality \eqref{eq:particle_bounds} is a direct consequence of $v^i_t \leq v_0^*$ and the first inequality from \eqref{eq:value_bounds}.

    (iv) For given $t \in (0, t')$, let $(w^i_t)_{i = 0}^{N'}$ be a subset of $(v_t^i)_{i = 0}^N$ such that either
    \begin{equation*}
        w_t^{j-1} \leq w_t^j \geq w_t^{j+1} \qquad \text{or} \qquad w_t^{j-1} \geq w_t^j \leq w_t^{j+1}
    \end{equation*}
    for all $i \in \set{1, \dots, N'-1}$, and moreover
    \begin{equation*} \label{eq:tv_simplification}
        |v(t)|_{\BV(\R)} = \sum_{i = 0}^{N-1} |v_t^{i+1} - v_t^i| = \sum_{i = 0}^{N'-1} |w_t^{i+1} - w_t^i|.
    \end{equation*}
    In other words, the tuple $(w^i_t)$ is a collection of peaks and troughs of $(v_t^i)$. Suppose further that $w_t^{i} \neq w_t^{i+1}$ for all $i \in \set{1, \dots, N'-1}$ (otherwise, one of the values could be removed from the tuple without changing the total variation). Let us further separate the index set $\{1, \dots, N'\}$ into two sets $P = \set{i\colon w_t^i < w_t^{i+1}}$ and $N = \set{i\colon w_t^i > w_t^{i+1}}$. Then the total variation of $v$ is
    \begin{equation*}
        |v(t)|_{\BV(\R)} = \sum_{i \in P} \bigl(w_t^{i+1} - w_t^i\bigr) - \sum_{i \in N} \bigl(w_t^{i+1} - w_t^i\bigr).
    \end{equation*}
    In view of \eqref{eq:max_principle} we have
    \begin{equation*}
        \frac{d}{dt} |v(t)|_{\BV(\R)} = \sum_{i \in P} \bigl(\dot{w}_t^{i+1} - \dot{w}_t^i\bigr) - \sum_{i \in N} \bigl(\dot{w}_t^{i+1} - \dot{w}_t^i\bigr) \leq 0,
    \end{equation*}
    which proves \eqref{eq:tv_dinimishing} upon integrating over $(0, t)$.
\end{proof}

\subsection{Existence and uniqueness} \label{subsec:existence_uniqueness}

Existence and uniqueness for \eqref{eq:particle_ode}--\eqref{eq:value_ode} depend on the regularity of the velocity function $V$. It is not Lipschitz on all of $\R$, so we restrict the class of solutions. For $\delta > 0$, let $\S_\delta^N$ be the collection of all tuples $(x^i)_{i = 1}^N$ and $(v^i)_{i = 1}^{N-1}$ such that
\begin{equation*}
    x^i < x^{i+1} \qquad \text{and} \qquad v^i = 0 \ \ \text{or}\ \ \delta \leq v^i \leq \frac{1}{\delta}
\end{equation*}
for all $i \in \set{1, \dots, N-1}$. Using the a priori properties from the previous section, we will show that local solutions of \eqref{eq:particle_ode}--\eqref{eq:value_ode} remain in $\S_\delta^N$ for some $\delta > 0$ which can be chosen independently of the interval of existence. The following lemma proves that $V$ is Lipschitz in this setting.

\begin{lemma} \label{lemma:velocity_regularity}
    Assume that $f$ is Lipschitz on $\R$. Then the particle velocity function $(v_l, v_r) \mapsto V(v_l, v_r)$ is locally Lipschitz on $(0, \infty)^2$, and similarly the functions ${v_r \mapsto V(0, v_r)}$ and $v_l \mapsto V(v_l, 0)$ are locally Lipschitz on $(0, \infty)$.
\end{lemma}

\begin{proof}
    Since $f$ is Lipschitz on $\R$, the function $v \mapsto a(v) = f(v)/v$ is at least locally Lipschitz on $(0, \infty)$. The particle velocity $V(v_l, v_r)$ is defined using the minimum or the maximum of $a(v)$ over the interval between $v_l$ and $v_r$ (see~\eqref{eq:velocity}). Since $\min$ and $\max$ operations preserve Lipschitz continuity, the function $V$ inherits the local Lipschitz property of $a$ when both $v_l, v_r$ are bounded away from zero. When one argument is zero, e.g. $v_r \mapsto V(0, v_r)$, this equals $\min_{v \in [0, v_r]} a(v)$. Thus, the rate of change of $V(0, v_r)$ is bounded by the rate of change of $a$ around $v_r$. A similar argument applies to $v_l \mapsto V(v_l, 0) = \max_{v \in [0, v_l]} a(v)$. A detailed case analysis easily confirms these claims.
\end{proof}

With the Lipschitz regularity of the velocity function established, we can now prove the main existence and uniqueness result for the particle path scheme. The following theorem ensures that the scheme produces a unique solution, valid globally in time until the first collision of particles.

\begin{theorem} \label{thm:existence_uniqueness}
    For any set of initial conditions $(x_0^i)_{i = 1}^N$ and $(v_0^i)_{i = 1}^{N-1}$ generated as in \mbox{\ref{item:1}}, there exists a unique set of solutions $(x_t^i)_{i = 1}^{N}$ and $(v_t^i)_{i = 1}^{N-1}$ of the system \eqref{eq:particle_ode}--\eqref{eq:value_ode} on the interval $(0, t_1)$, where $t_1 \in (0, T]$ is a uniquely determined collision time given by \eqref{eq:collision_time}.
\end{theorem}

\begin{proof}
    We will prove that for any $(x_0^i)$, $(v_0^i)$ in $S^N_{\delta_0}$, there is a unique solution $(x_t^i)$, $(v_t^i)$ of \eqref{eq:particle_ode}--\eqref{eq:value_ode} on $(0, t_1)$ which belongs to $S^N_{\delta_1}$, where
    \begin{equation} \label{eq:delta_1}
        0 \leq \delta_1 = \delta_0 \min_{i \in \set{1, N-1}} \frac{\dx_0^i}{\dx_0^i + T(a_{\max} - a_{\min})}.
    \end{equation}
    To this end, consider $(x_0^i)$, $(v_0^i)$ in $S^N_{\delta_0}$ for some $\delta_0 > 0$. The equations \eqref{eq:particle_ode}--\eqref{eq:value_ode} can be written compactly as
    \begin{equation*}
        \dot{x}_t^i = V(v_t^{i-1}, v_t^i) = V\biggl(\frac{\dx_0^{i-1}}{\dx_t^{i-1}} v_0^{i-1}, \frac{\dx_0^i}{\dx_t^i} v_0^i, \biggr) \eqqcolon F^i\bigl((x_t^i)_{i=1}^N; (x_0^i)_{i=1}^N, (v_0^i)_{i=1}^{N-1}\bigr)
    \end{equation*}
    for $i \in \set{1, \dots, N}$. The right-hand side is a function $F^i\colon \R^N \to \R^N$ which takes parameters $(x_0^i)$ and $(v_0^i)$. In light of Lemma~\ref{lemma:velocity_regularity}, we see that $F^i$ is locally Lipschitz around the initial condition $(x_0^i)$ for all $i \in \set{1, \dots, N}$. This gives a unique local solution $(x_t^i)$, $(v_t^i)$ on an interval $(0, t')$ for some $t' > 0$. Since $x_0^1 < x_0^2 < \dots < x_0^N$ and the solution is continuous with respect to time, we may assume that $x_t^1 < x_t^2 < \dots < x_t^N$ for all $(0, t')$. Moreover, by (ii) from Proposition~\ref{prop:a_priori_properties} the local densities $(v_t^i)$ which are nonzero are bounded from below by a constant as in \eqref{eq:delta_1}, independently of the existence time $t'$. Thus, the local solution belongs to $S^N_{\delta_1}$.
    
    To obtain a global solution, it suffices to iterate this argument finitely many times up to a possible collision time $t_1$ when $x_{t_1}^i = x_{t_1}^{i+1}$ for some $i \in \set{1, \dots, N}$.
\end{proof}

We have so far considered solutions only up to the first collision time $t_1$, but the results of this section extend to the whole interval $[0, T]$. In particular, the bounds \eqref{eq:max_principle} hold for all $t \in [0, T] \setminus \mathcal{T}_c$, which implies that the properties from Proposition~\ref{prop:a_priori_properties} are valid for all $t \in [0, T]$. Since particles and local densities are removed at collision times, the total variation after a collision can only have decreased, guaranteeing that the total variation estimate \eqref{eq:tv_dinimishing} holds globally. Finally, an analogue of Theorem~\ref{thm:existence_uniqueness} holds on each interval $(t_j, t_{j+1})$, ensuring existence and uniqueness of a global solution of \eqref{eq:particle_ode}--\eqref{eq:value_ode}.

\section{PDE formulation} \label{sec:pde_formulation}

In this section, we will show that the approximation $v$ from \eqref{eq:approx_sol_global} is the weak solution of a continuity equation of the form \eqref{eq:cont_eq}, with a velocity field $A$ given by the linear interpolation of particle velocities. Ultimately, this formulation allows us to prove that $v$ satisfies an approximate entropy inequality.

Recall from Section~\ref{sec:existence_uniqueness} that the particle path scheme yields a unique global solution defined on $[0, T]$. This solution consists of a finite set of collision times $(t_j)_{j \geq 0} \subset [0, T]$ with corresponding particle paths $(x_t^\pi)$ and local densities $(v_t^\pi)$ evolving according to \eqref{eq:particle_ode}--\eqref{eq:value_ode} between collisions.

\subsection{Approximate continuity equation} \label{subsec:continuity_equation}

We begin by showing that $v$ is a weak solution of a continuity equation. Define
\begin{equation} \label{eq:velocity_interpolation_global}
    \begin{aligned}
        A\bigl((x_t^\pi), (v_t^\pi); x, t\bigr) & \coloneqq \frac{x_t^{\pi_j(i+1)}-x}{x_t^{\pi_j(i+1)} - x_t^{\pi_j(i)}} V\bigl(v_t^{\pi_j(i-1)}, v_t^{\pi_j(i)}\bigr) \\
        & \quad + \frac{x-x_t^{\pi_j(i)}}{x_t^{\pi_j(i+1)} - x_t^{\pi_j(i)}} V\bigl(v_t^{\pi_j(i)}, v_t^{\pi_j(i+1)}\bigr)
    \end{aligned}
\end{equation}
for $t \in (t_j, t_{j+1})$ and $x \in \bigl[x_t^{\pi_j(i)}, x_t^{\pi_j(i+1)}\bigr]$. We will mainly write this as $A(x, t)$, but included $(x^\pi)$ and $(v^\pi)$ as arguments in the definition above to stress this dependence. The function $A$ is a global-in-time piecewise linear interpolation of particle velocities $(v^\pi)$, with interpolation nodes given by particle positions $(x^\pi)$.

\begin{proposition} \label{prop:weak_solution_cont_eq}
    Let $v$ be constructed by the particle path scheme and defined as in \eqref{eq:approx_sol_global}. Then it satisfies
    \begin{equation} \label{eq:weak_continuity_unshifted}
        \int_0^T \int_\R \bigl(\partial_t \varphi(x, t) + A(x, t) \partial_x \varphi(x, t) \bigr) v(x, t)\, dx\,dt + \int_\R \varphi(x, 0) v_0(x)\, dx = 0
    \end{equation}
    for all $\varphi \in C^\infty_c(\R \times [0, T))$.
\end{proposition}

\begin{remark}
    We will treat \eqref{eq:weak_continuity_unshifted} as a linear equation with a prescribed velocity $A$. However, since both $A$ and $v$ depend on the underlying output $(x_t^\pi)$ and $(v_t^\pi)$ of the particle path scheme, the equation can also be viewed as nonlinear. From this perspective, it is also a nonlocal equation, since $A$ is a piecewise linear interpolation with interpolation nodes $(x_t^\pi)$.
\end{remark}

\begin{proof}
    There can only be finitely many collisions of particles, so it suffices to prove that the equation holds for $v$ given by \eqref{eq:approx_sol_local} on the interval $(0, t_1)$, for all test functions $\varphi \in C^\infty_c(\R \times [0, t_1])$. To that end, let $\varphi \in C^\infty_c(\R \times [0, t_1])$ such that $\varphi = \gamma \vartheta$ for $\vartheta \in C^\infty_c(\R)$ and $\gamma \in C^\infty_c([0, t_1])$. Then
    \begin{equation} \label{eq:cont_eq_split}
        \begin{aligned}
            \int_0^{t_1} \int_\R v \partial_t \varphi\, dx\,dt & = \sum_{i=0}^{N} \int_0^{t_1} v_t^i \partial_t \gamma(t) \biggl(\int_{x_t^i}^{x_t^{i+1}} \vartheta(x)\, dx \biggr)\, dt \\
            & = \sum_{i = 0}^N \gamma(t_1) v_{t_1}^i \int_{x_{t_1}^i}^{x_{t_1}^{i+1}} \vartheta(x)\, dx - \sum_{i = 0}^N\gamma(0) v_0^i \int_{x_0^i}^{x_0^{i+1}} \vartheta(x)\, dx \\
            & \quad - \sum_{i = 0}^N \int_0^{t_1} \gamma(t) \partial_t \biggl(v_t^i \int_{x_t^i}^{x_t^{i+1}} \vartheta(x)\, dx \biggr)\, dt
        \end{aligned}
    \end{equation}
    where we have used integration by parts in the last equality. Adopting the shorthand notation $V_t^i = V(v_t^{i-1}, v_t^{i})$, the temporal derivative in the last term can be written as
    \begin{equation*}
        \begin{aligned}
            & \partial_t \biggl(v_t^i \int_{x_t^i}^{x_t^{i+1}} \vartheta(x)\, dx\biggr) \\
            & = -v_t^i \frac{V_t^{i+1} - V_t^i}{x_t^{i+1} - x_t^i} \int_{x_t^i}^{x_t^{i+1}} \vartheta(x)\, dx + v_t^i \bigl[\vartheta(x_t^{i+1}) V_t^{i+1} - \vartheta(x_t^i) V_t^i \bigr] \\
            & = v_t^i \biggl[V_t^i \biggl(\int_{x_t^i}^{x_t^{i+1}} \frac{\vartheta(x)}{x_t^{i+1} - x_t^{i}}\, dx - \vartheta(x_t^i) \biggr) - V_t^{i+1} \biggl(\int_{x_t^i}^{x_t^{i+1}} \frac{\vartheta(x)}{x_t^{i+1} - x_t^{i}}\, dx - \vartheta(x_t^{i+1}) \biggr)\biggr] \\
            & = v_t^i \biggl( V_t^i \int_{x_t^i}^{x_t^{i+1}} \frac{x_t^{i+1} - x}{x_t^{i+1} - x_{i}} \partial_x \vartheta(x)\, dx + V_t^{i+1} \int_{x_t^i}^{x_t^{i+1}} \frac{x - x_t^i}{x_t^{i+1} - x_{i}} \partial_x \vartheta(x)\, dx\biggr) \\
            & = v_t^i \int_{x_t^i}^{x_t^{i+1}} A(x, t) \partial_x \vartheta(x)\, dx,
        \end{aligned}
    \end{equation*}
    for all $i \in \set{1, \dots, N-1}$ (recall that $v_t^0 = v_t^N = 0$ for all $t$). Here, we have used \eqref{eq:particle_ode} in the first equality and integration by parts in the third equality. Inserting this in \eqref{eq:cont_eq_split} yields
    \begin{equation*}
        \begin{aligned}
            \int_0^{t_1} \int_\R v \partial_t \varphi\, dx\,dt & = \sum_{i = 0}^N \gamma(t_1) v_{t_1}^i \int_{x_{t_1}^i}^{x_{t_1}^{i+1}} \vartheta(x)\, dx - \sum_{i = 0}^N\gamma(0) v_0^i \int_{x_0^i}^{x_0^{i+1}} \vartheta(x)\, dx\\
            & \quad - \sum_{i = 0}^N \int_0^{t_1} \gamma(t) v_t^i \biggl(\int_{x_t^i}^{x_t^{i+1}} A(x, t) \partial_x \vartheta(x)\, dx \biggr) dt \\
            & = \int_\R v(t_1) \varphi(t_1)\, dx - \int_\R v_0 \varphi(0)\, dx - \int_0^{t_1} \int_{\R} v A \partial_x \varphi\, dx\,dt.
        \end{aligned}
    \end{equation*}
    The general case now follows by approximation of arbitrary $\varphi \in C_c^\infty(\R \times [0, t_1])$ by finite linear combinations of $\gamma(t) \vartheta(x)$.
\end{proof}

Proposition~\ref{prop:weak_solution_cont_eq} implies some temporal regularity of the function $v$.

\begin{lemma} \label{lemma:temporal_modulus}
    Let $v$ be given by \eqref{eq:approx_sol_global}. Then
    \begin{equation} \label{eq:temporal_continuity}
        \norm{v(t) - v(s)}_{L^1(\R)} \leq 4 [f]_\lip |v_0|_{\BV(\R)} |t-s|
    \end{equation}
    for all $s, t \in [0, T]$.
\end{lemma}

\begin{proof}
    Since $v$ satisfies \eqref{eq:weak_continuity_unshifted}, for all $\vartheta \in C^\infty_c(\R)$ the function $t \mapsto \int_{\R} v \vartheta\, dx$ is absolutely continuous on $(0, T)$ with weak derivative $\int_{\R} A v \partial_x \vartheta\, dx$. Equivalently stated,
    \begin{equation*}
        \int_{\R} v(x, t) \vartheta(x)\, dx - \int_{\R} v(x, s) \vartheta(x)\, dx = \int_s^{t} \int_{\R} A(x, r) v(x, r) \partial_x \vartheta(x) \,dx\,dr
    \end{equation*}
    for all $t < s$. Using integration by parts on the right-hand side and taking the supremum over all $|\vartheta| \leq 1$ yields
    \begin{equation*}
        \int_{\R} \bigl|v(x, t) - v(x, s)\bigr| \, dx = \int_s^{t} \int_{\R} \bigl|\partial_x \bigl(A(x, r) v(x, r)\bigr)\bigr| \,dx\,dr,
    \end{equation*}
    from which we infer
    \begin{equation*}
        \norm{v(t) - v(s)}_{L^1(\R)} \leq \sup_{r \in (s, t)} |A(r) v(r)|_{\BV(\R)} |t-s|.
    \end{equation*}
    It remains to estimate the total variation of $Av$. It is a piecewise linear function with breakpoints at $(x^i)$ and values
    \begin{equation*}
        \lim_{x \to (x_t^i)^-} A(x, t) v(x, t) = a(\tilde{v}_t^i) v_t^{i-1}, \qquad \lim_{x \to (x_t^i)^+} A(x, t) v(x, t) = a(\tilde{v}_t^i) v_t^{i},
    \end{equation*}
    for all $i \in \set{1, \dots, N}$, where $\tilde{v}_t^i \in [v_t^{i-1}, v_t^i]$ is defined as the point at which $a$ attains the value $V(v_t^{i-1}, v_t^{i})$ (note that $\tilde{v}_t^i$ is not necessarily unique). Thus, the total variation is
    \begin{equation} \label{eq:tv_estimate_start}
            |A(t) v(t)|_{\BV(\R)} = \sum_{i = 1}^N |a(\tilde{v}_t^{i})v_t^{i-1} - a(\tilde{v}_t^{i})v_t^{i}| + \sum_{i = 1}^{N-1} |a(\tilde{v}_t^{i})v_t^{i} - a(\tilde{v}_t^{i+1})v_t^{i}|.
    \end{equation}
    Since 
    \begin{equation*}
        |a(v)| = \frac{|f(v)|}{|v|} = \frac{|f(v) - f(0)|}{|v - 0|}\leq [f]_{\lip}
    \end{equation*}
    for all $v \in \R$, the first sum in \eqref{eq:tv_estimate_start} can be bounded by
    \begin{equation*}
        \sum_{i = 1}^N |a(\tilde{v}_t^{i})v_t^{i-1} - a(\tilde{v}_t^{i})v_t^{i}| \leq [f]_{\lip} \sum_{i = 1}^N |v_t^{i-1} - v_t^{i}| \leq [f]_{\lip} |v(t)|_{\BV(\R)}.
    \end{equation*}
    For the terms in the second sum of \eqref{eq:tv_estimate_start}, we have
    \begin{equation*}
        \begin{aligned}
            & |a(\tilde{v}_t^{i})v_t^{i} - a(\tilde{v}_t^{i+1})v_t^{i}| \\
            & \quad \leq |a(\tilde{v}_t^{i})v_t^{i} - a(\tilde{v}_t^{i})\tilde{v}_t^{i}| + |a(\tilde{v}_t^{i})\tilde{v}_t^{i} - a(\tilde{v}_t^{i+1})\tilde{v}_t^{i+1}| + |a(\tilde{v}_t^{i+1})\tilde{v}_t^{i+1} - a(\tilde{v}_t^{i+1})v_t^{i}| \\
            & \quad \leq |a(\tilde{v}_t^{i})||v_t^{i} - \tilde{v}_t^{i}| + |f(\tilde{v}_t^{i}) - f(\tilde{v}_t^{i+1})| + |a(\tilde{v}_t^{i+1})||\tilde{v}_t^{i+1} - v_t^{i}| \\
            & \quad \leq [f]_{\lip} \bigl(|v_t^{i} - \tilde{v}_t^{i}| + |\tilde{v}_t^{i} - \tilde{v}_t^{i+1}| + |\tilde{v}_t^{i+1} - v_t^{i}| \bigr),
        \end{aligned}
    \end{equation*}
    which gives
    \begin{equation*}
        \begin{aligned}
            \sum_{i = 1}^{N-1} |a(\tilde{v}_t^{i})v_t^{i} - a(\tilde{v}_t^{i+1})v_t^{i}| & \leq [f]_\lip \sum_{i = 1}^{N-1} \bigl(|v_t^{i} - \tilde{v}_t^{i}| + |\tilde{v}_t^{i} - \tilde{v}_t^{i+1}| + |\tilde{v}_t^{i+1} - v_t^{i}| \bigr) \\
            & \leq [f]_\lip \sum_{i = 1}^{N-1} \bigl(|v_t^{i} - v_t^{i-1}| + |\tilde{v}_t^{i} - \tilde{v}_t^{i+1}| + |v_t^{i+1} - v_t^{i}| \bigr) \\
            & \leq 3 [f]_\lip |v(t)|_{\BV(\R)}.
        \end{aligned}
    \end{equation*}
    Plugging this into \eqref{eq:tv_estimate_start} and using that the total variation of $v$ is nonincreasing from part (iv) of Proposition~\ref{prop:a_priori_properties}, we obtain \eqref{eq:temporal_continuity}.
\end{proof}

Towards an entropy-like inequality for $v$, we consider the vertically shifted function $v_k \coloneqq v-k$ for constants $k \in \R$. Set also $v_{0, k} \coloneqq v_0 - k$. Since $v$ is a weak solution of \eqref{eq:weak_continuity_unshifted}, it follows easily that $v_k$ is a weak solution to
\begin{equation} \label{eq:cont_eq_shifted}
    \left\{
    \begin{aligned}
        & \partial_t v_k + \partial_x \bigl(A v_k\bigr) = -k \partial_x A \\
        & v_k(0) = v_{0, k}
    \end{aligned}
    \right.
\end{equation}
in the sense of \eqref{eq:weak_general_continuity}. Note that $\partial_x A$ is finite except at collision times $t_j \in \mathcal{T}_c$, and satisfies
\begin{equation*}
    \esssup_{t \in (0, T)} \int_\R |\partial_x A(x, t)|\, dx  = \esssup_{t \in (0, T)} |A(t)|_{\BV(\R)} \leq (N+1) a_{\max},
\end{equation*}
where we recall that $a_{\max} = \max_{v \in [0, v_0^*]} a(v)$. The next lemma shows that the equation \eqref{eq:cont_eq_shifted} can be studied within the one-sided Lipschitz framework from Section~\ref{subsec:continuity_eq}.

\begin{lemma}
    The function $x \mapsto A(x, t)$ is continuous on $\R$ for all $t \in [0, T] \setminus \mathcal{T}_c$. Furthermore, it is one-sided Lipschitz (cf.~\eqref{eq:one_sided_lip}) for all $t \in [0, T]$.
\end{lemma}

\begin{proof}
    Since $x \mapsto A(x, t)$ is a linear interpolation of values $V(v^{i-1}, v^i)$ on interpolation nodes $(x_t^i)$, it is Lipschitz except possibly at collision times, towards which two particles move arbitrarily close. Consider without loss of generality the interval $[0, t_1]$. From \eqref{eq:particle_bounds}, two particles $x^i$ and $x^{i+1}$ can only collide if $v^i = 0$. But then
    \begin{equation*}
        A(x_t^{i+1}, t) = V(0, v_t^{i+1}) = \min_{v \in [0, v_t^{i+1}]} a(v) \leq \max_{v \in [0, v_t^{i-1}]} a(v) = V(v_t^{i-1}, 0) = A(x_t^{i}, t),
    \end{equation*}
    for all $t \in [0, t_1]$, and the one-sided Lipschitz condition \eqref{eq:one_sided_lip} holds.
\end{proof}

Consider the system of ODEs
\begin{equation} \label{eq:flow}
    \left\{
    \begin{alignedat}{2}
        & \dot{x}_t = A(x_t, t) \qquad && \text{for a.e.}\ t \in (s, T) \\
        & x_s = x.
    \end{alignedat}
    \right.
\end{equation}
Using the theory for continuity equations outlined in Section~\ref{subsec:continuity_eq}, we obtain a representation formula for $v_k$.

\begin{proposition} \label{prop:representation}
    The system of ODEs \eqref{eq:flow} generates a unique forward Lipschitz flow $X = X_t(x, s)$. Let $v$ be defined by \eqref{eq:approx_sol_global}. Then for all $k \in \R$, the function $v_k = v-k$ is the unique weak solution of \eqref{eq:cont_eq_shifted}. Moreover, for all $s \in [0, T]$, it satisfies
    \begin{equation*} \label{eq:representation_k}
        v_k(t) = (X_t(s))_\#v_k(s) - k \int_s^t X_t(r)_{\#} (\partial_x A(r))\, dr
    \end{equation*}
    on $\R \times (s, T)$.
\end{proposition}

\subsection{Entropy inequality} \label{subsec:entropy_inequality}

We prove that $v$ satisfies the approximate entropy inequality \eqref{eq:strong_approx_entropy_ineq}, a key step in our analysis. The inequality in the following lemma turns out to be the cornerstone of this proof. It does not hold in general for any other choice of particle velocity other than \eqref{eq:velocity}. In this sense, it confirms that \eqref{eq:velocity} is the correct entropic velocity. Recall that we denote $v - k$ by $v_k$.

\begin{lemma} \label{lemma:entropic_velocity_step}
    Let $v$ be given by \eqref{eq:approx_sol_global}. Then
    \begin{equation} \label{eq:entropic_velocity_step}
        \begin{aligned}
            & \int_{\R} \vartheta(x) \sgn(v_k(x, t)) \partial_x (A(x, t) - a(k))\, dx \\
            & \quad \geq - \int_{\R} \partial_x \vartheta(x) \sgn(v_k(x, t)) (A(x, t) - a(k))\, dx
        \end{aligned}
    \end{equation}
    for all $t \in (0, T) \setminus \mathcal{T}_c$, all $k \in \R$ and all nonnegative $\vartheta \in C^\infty_c(\R)$.
\end{lemma}

\begin{proof}
    Let $t \in (0, t_1)$ without loss of generality. Splitting the integral on the left-hand side of \eqref{eq:entropic_velocity_step} and using integration by parts, we see that
    \begin{equation} \label{eq:entropic_velocity_integral_split}
        \begin{aligned}
            & \int_{\R} \vartheta(x) \sgn(v_k(x, t)) \partial_x \bigl(A(x, t) - a(k)\bigr) \,dx \\
            & \quad = \sum_{i = 0}^N \int_{x_t^i}^{x_t^{i + 1}} \vartheta(x) \sgn(v_t^i - k) \partial_x \bigl(A(x, t) - a(k)\bigr) \,dx \\
            & \quad = \sum_{i = 0}^N \sgn(v_t^i - k) \biggl[\vartheta(x_t^{i+1}) \bigl(A(x_t^{i + 1}, t) - a(k)\bigr) - \vartheta(x_t^i) \bigl(A(x_t^i, t) - a(k)\bigr)\biggr] \\
            & \qquad - \sum_{i = 0}^N \int_{x_t^i}^{{x_t^{i + 1}}} \partial_x \vartheta(x) \sgn(v_t^i - k) \bigl(A(x, t) - a(k)\bigr)\, dx.
        \end{aligned}
    \end{equation}
    In the third line, the series is telescoping except when $v_t^i - k$ and $v_t^{i + 1} - k$ have different signs. This becomes apparent when writing
    \begin{equation*}
        \begin{aligned}
            & \sum_{i = 0}^N \sgn(v_t^i - k) \biggl[\vartheta(x_t^{i + 1}) \bigl(A(x_t^{i + 1}, t) - a(k)\bigr) - \vartheta(x_t^{i}) \bigl(A(x_t^{i}, t) - a(k)\bigr)\biggr] \\
            & \quad = \sum_{i = 0}^{N-1} \bigl(\sgn(v_t^i - k) - \sgn(v_t^{i + 1} - k)\bigr) \vartheta(x_t^{i+1}) \bigl(A(x_t^{i+1}, t) - a(k)\bigr),
        \end{aligned}
    \end{equation*}
    where we used $\vartheta(x_t^0) = \vartheta(x_t^{N+1}) = 0$. If $\sgn(v_t^i-k) < \sgn(v_t^{i+1}-k)$, then $v_t^{i} \leq k \leq v_t^{i+1}$, and consequently
    \begin{equation*}
        A(x_t^{i+1}, t) = \min_{v \in [v_t^i, v_t^{i + 1}]} a(v) \leq a(k).
    \end{equation*}
    On the other hand, if $\sgn(v_t^i-k) > \sgn(v_t^{i+1}-k)$, then $v_t^{i} \geq k \geq v_t^{i+1}$, and
    \begin{equation*}
        A(x_t^{i+1}, t) = \max_{v \in [v_t^{i+1}, v_t^{i}]} a(v) \geq a(k).
    \end{equation*}
    In both cases, the inequality
    \begin{equation*}
        \bigl(\sgn(v_t^i-k) - \sgn(v_t^{i + 1}-k)\bigr) \vartheta(x_t^{i+1}) \bigl(A(x_t^{i+1}, t) - a(k)\bigr) \geq 0
    \end{equation*}
    holds. Inserting this in \eqref{eq:entropic_velocity_integral_split} yields \eqref{eq:entropic_velocity_step}.
\end{proof}

Next, we use the previous lemma to prove the approximate entropy inequality.

\begin{theorem}
    The function $v_k = v - k$, where $v$ is given by \eqref{eq:approx_sol_global}, satisfies
    \begin{equation} \label{eq:weak_approx_entropy_ineq}
        \begin{aligned}
            & \int_0^T \int_\R |v_k| \partial_t \varphi + (Av - f(k)) \sgn(v_k) \partial_x \varphi \,dx\,dt \\
            & \quad - \int_\R \varphi(T) |v_k(T)|\, dx + \int_\R \varphi(0) |v_{0, k}|\, dx \geq 0
        \end{aligned}
    \end{equation}
    for all $k \geq 0$ and nonnegative $\varphi \in C^\infty_c(\R \times [0, T])$.
\end{theorem}

\begin{proof}
    Let $ 0 \leq \varphi \in C^\infty_c(\R \times [0, T])$ and set $\eps > 0$. We begin by writing
    \begin{equation} \label{eq:weak_approx_entropy_ineq_start}
        \begin{aligned}
            \int_0^T \int_{\R} \varphi |v_k| \,dx\,dt & = \int_0^\eps \int_{\R} \varphi |v_k|\, dx\,dt \\
            & \quad + \int_0^{T-\eps}\! \int_{\R} \varphi(t + \eps) \sgn(v_k(t+\eps)) v_k(t + \eps)\,dx\,dt.
        \end{aligned}
    \end{equation}
    For the last term, we use the formula
    \begin{equation*}
        v_k(t+\eps) = (X_{t+\eps}(t))_\#v_k(t) - k \int_t^{t + \eps} (X_{t+\eps}(r))_{\#} (\partial_x A(r))\, dr
    \end{equation*}
    from Proposition~\ref{prop:representation}, which gives
    \begin{equation*}
        \begin{aligned}
            & \int_0^{T-\eps}\! \int_{\R} \varphi(t + \eps) \sgn(v_k(t + \eps)) v_k(t + \eps)\,dx\,dt \\
            & \quad = \int_0^{T-\eps}\! \int_{\R} \varphi\bigl(X_{t+\eps}(t), t + \eps\bigr) \sgn\bigl(v_k(X_{t+\eps}(t), t + \eps)\bigr) v_k(t)\, dx\,dt \\
            & \qquad - k \int_0^{T-\eps}\! \int_t^{t + \eps}\! \int_{\R} \varphi\bigl(X_{t+\eps}(t), t + \eps\bigr) \sgn\bigl(v_k(X_{t+\eps}(r), t + \eps)\bigr) \partial_x A(r) \,dx\,dr\,dt
        \end{aligned}
    \end{equation*}
    (see \cite[Definition 1.70]{ambrosio_fusco_pallara_2000} for justification of using the push-forward formula \eqref{eq:push_forward} in this setting). In the above integrals, we have dropped the spatial argument for readability. Since
    \begin{equation*}
        \begin{aligned}
            & \int_0^{T-\eps}\! \int_{\R} \varphi\bigl(X_{t+\eps}(t), t + \eps\bigr) \sgn\bigl(v(X_{t+\eps}(t), t + \eps)\bigr) v_k(t)\, dx\,dt \\
            & \quad \leq \int_0^{T-\eps}\! \int_{\R} \varphi\bigl(X_{t+\eps}(t), t + \eps\bigr) |v_k(t)| \,dx\,dt,
        \end{aligned}
    \end{equation*}
    inserting the above in \eqref{eq:weak_approx_entropy_ineq_start} yields
    \begin{equation*}
        \begin{aligned}
            & \int_0^T \int_{\R} \varphi |v_k| \,dx\,dt \leq \int_0^\eps \int_{\R} \varphi |v_k| \,dx\,dt + \int_0^{T-\eps}\! \int_{\R} \varphi\bigl(X_{t+\eps}(t), t + \eps\bigr) |v_k(t)| \,dx\,dt \\
            & \qquad - k \int_0^{T-\eps}\! \int_t^{t + \eps}\! \int_{\R} \varphi\bigl(X_{t+\eps}(r), t + \eps\bigr) \sgn\bigl(v_k(X_{t+\eps}(r), t + \eps)\bigr) \partial_x A(r) \,dx\,dr\,dt.
        \end{aligned}
    \end{equation*}
    Dividing by $\eps$ and rearranging the terms, we have
    \begin{equation} \label{eq:finite_entropy_difference}
        \begin{aligned}
            & \int_0^{T-\eps}\! \int_{\R} \frac{\varphi\bigl(X_{t+\eps}(x, t), t + \eps\bigr) - \varphi(x, t)}{\eps} |v_k(x, t)| \,dx\,dt \\
            & \quad - \frac{1}{\eps} \int_{T-\eps}^{T}\! \int_{\R} \varphi(x, t) |v_k(x, t)|\, dx\,dt + \frac{1}{\eps} \int_0^{\eps} \int_{\R} \varphi(x, t) |v_k(x, t)|\, dx\,dt \\
            & \geq \frac{k}{\eps} \int_0^{T-\eps}\! \int_t^{t + \eps}\! \int_{\R} \varphi\bigl(X_{t+\eps}(x, r), t + \eps\bigr) \sgn\bigl(v_k(X_{t+\eps}(x, r), t + \eps)\bigr) \partial_x A(x, r) \,dx\,dr\,dt.
        \end{aligned}
    \end{equation}
    Towards passing $\eps \to 0$, note that
    \begin{equation} \label{eq:finite_flow_difference}
        \begin{aligned}
            & \frac{\varphi\bigl(X_{t+\eps}(x, t), t + \eps\bigr) - \varphi(x, t)}{\eps} \\
            & \quad = \int_0^1 \partial_t\varphi(x, t+r\eps)+ \partial_x \varphi\big(x+r(X_{t+\eps}(x,t) - x), t + \eps\big) \bigg(\frac{X_{t+\eps}(x, t) - x}{\eps}\bigg) \,dr.
        \end{aligned}
    \end{equation}
    Since furthermore
    \begin{equation*}
        \frac{X_{t+\eps}(x, t) - x}{\eps} = \frac{1}{\eps} \int_t^{t+\eps} A(X_r(x, t), r)\, dr
    \end{equation*}
    and the velocity $A$ is continuous for all $t \in [0, T] \setminus \mathcal{T}_c$, we see that \eqref{eq:finite_flow_difference} converges to $\partial_t \varphi + A \partial_x \varphi$ in $L^1(\R \times (0, T))$ as $\eps \to 0$. Taking the limit in \eqref{eq:finite_entropy_difference} therefore yields
    \begin{equation*}
        \begin{aligned}
            & \int_0^{T} \int_{\R} \bigl(\partial_t \varphi + A \partial_x \varphi\bigr) |v_k| \,dx\,dt - \int_{\R} \varphi(T) |v_k(T)|\, dx + \int_{\R} \varphi(0) |v_{0, k}| \,dx\,dt \\
            & \quad \geq k \int_0^{T} \int_{\R} \varphi \sgn(v_k) \partial_x A \,dx\,dt.
        \end{aligned}
    \end{equation*}
    We now insert $\partial_x A = \partial_x(A - a(k))$ and use the result of Lemma~\ref{lemma:entropic_velocity_step}, by which we get
    \begin{equation*}
        \begin{aligned}
            & \int_0^{T} \int_{\R} \bigl(\partial_t \varphi + A \partial_x \varphi\bigr) |v_k| \,dx\,dt - \int_{\R} \varphi(T) |v_k(T)|\, dx + \int_{\R} \varphi(0) |v_{0, k}|\,dx\,dt \\
            & \quad \geq - k \int_0^{T} \int_{\R} \partial_x \varphi \sgn(v_k) (A - a(k)) \,dx\,dt
        \end{aligned}
    \end{equation*}
    for all nonnegative $k \in \R$. Finally, in view of
    \begin{equation*}
        A |v_k| + k \sgn(v_k)(A - a(k)) = (A v - f(k)) \sgn(v_k),
    \end{equation*}
    we arrive at \eqref{eq:weak_approx_entropy_ineq}.
\end{proof}

\section{Convergence rate} \label{sec:convergence_rate}

By comparing \eqref{eq:weak_approx_entropy_ineq} to the exact entropy inequality, it is possible to estimate how far $v$ is from being an entropy solution of the conservation law \eqref{eq:conservation_law}. Following \cite[Chap.~3.3]{holden_risebro_2015}, we define the Kruzkhov form
\begin{equation*}
    \begin{aligned}
        \Lambda_T(v, \varphi, k) & = \int_0^T \int_\R |v_k| \partial_t \varphi + \bigl(f(v) - f(k)\bigr) \sgn(v_k) \partial_x \varphi\, dx\,dt \\
        & \quad - \int_\R |v_k(T)| \varphi(T)\, dx + \int_\R |v_{k, 0}| \varphi(0)\, dx,
    \end{aligned}
\end{equation*}
where $k \in \R$ and $\varphi \in C^\infty_c(\R \times [0, T])$. Let $(\rho_\eps)_{\eps \geq 0}$ be a family of standard mollifiers, and set
\begin{equation*}
    \Omega^{\eps, \tau}(x, y, t, s) = \rho^\eps(x-y) \rho^\tau(t-s).
\end{equation*}
Furthermore define
\begin{equation} \label{eq:double_kruzkov_form}
    \Lambda_T^{\eps, \tau}(v, u) \coloneqq \int_0^T \int_{\R} \Lambda_T(v(\cdot, \cdot), \Omega^{\eps, \tau}(\cdot, y, \cdot, s), u(y, s)) \,dy\,ds.
\end{equation}

To prove Theorem \ref{thm:main}, we will invoke Kuznetsov's lemma \cite[Theorem 3.14]{holden_risebro_2015}: If $u$ is the entropy solution of \eqref{eq:conservation_law} and $v$ is an approximation which satisfies
\begin{equation} \label{eq:kuznetsov_class}
    \norm{v(t)}_{L^\infty} \leq \norm{v_0}_{L^\infty}\quad \text{and}\quad |v(t)|_{\BV} \leq |v_0|_{\BV}
\end{equation}
for all $t \in [0, T]$, then
\begin{equation} \label{eq:kuznetsov}
    \begin{aligned}
        \norm{v(T) - u(T)}_{L^1(\R)} & \leq \norm{v_0 - u_0}_{L^1(\R)} + |u_0|_{\BV(\R)} (2\eps + \tau [f]_{\lip}) \\
        & \quad + \nu(v, \tau) - \Lambda^{\eps, \tau}(v, u)
    \end{aligned}
\end{equation}
for all $\eps > 0$ and $\tau \in (0, T)$, where
\begin{equation*}
    \nu(v, \tau) \coloneqq \sup_{t \in (0, T)} \sup_{r \in (0, \tau)} \norm{v(t + r) - v(t)}_{L^1(\R)}.
\end{equation*}
Note in particular that the function $v$ generated by the the particle path scheme satisfies the conditions in \eqref{eq:kuznetsov_class}.

\begin{proof}[Proof of Theorem~\ref{thm:main}]
    Let $v$ be the approximation generated by the particle path scheme as defined in \eqref{eq:approx_sol_global}. Using the approximate entropy inequality \eqref{eq:weak_approx_entropy_ineq}, we have
    \begin{equation*}
        \Lambda_T(v, \varphi, k) \geq \int_0^T \int_\R \bigl(f(v) - Av\bigr) \sgn(v_k)\partial_x \varphi\, dx\,dt
    \end{equation*}
    for all nonnegative $\varphi \in C^\infty_c(\R \times [0, T])$. Plugging this into \eqref{eq:double_kruzkov_form} gives
    \begin{equation} \label{eq:kuznetsov_start}
        \begin{aligned}
            \Lambda_T^{\eps, \tau} & \geq \int_0^T \int_\R \int_0^T \int_\R \bigl(f(v(y, s)) - A(y, s) v(y, s) \bigr) \\
            & \quad \times \sgn(v(y, s) - u(x, t)) \partial_x \Omega^{\eps, \tau}(x, y, t, s) \,dx\,dt\,dy\,ds.
        \end{aligned}
    \end{equation}
    In view of
    \begin{equation*}
        \nu(v, \tau) \leq 4 |\tau| [f]_{\lip} [v_0]_{BV} \leq 4 |\tau| [f]_{\lip} [u_0]_{BV}
    \end{equation*}
    by Lemma~\ref{lemma:temporal_modulus}, we can immediately pass $\tau \to 0$ to get rid of one of the temporal integrals. The remaining contribution to the right hand side of \eqref{eq:kuznetsov_start} can be estimated in the absolute value according to
    \begin{equation*}
        \begin{aligned}
            & \int_0^T \int_\R \int_\R \bigl|f(v(y, t)) - A(y, t) v(y, t) \bigr| \bigl|\partial_x \rho_{\eps}(x-y)\bigr|  \,dx\,dy\,dt \\
            & \quad \leq \frac{1}{\eps} \int_0^T \int_\R \bigl|f(v(y, t)) - A(y, t) v(y, t) \bigr| \,dy\,dt.
        \end{aligned}
    \end{equation*}
    Inserting this in \eqref{eq:kuznetsov} yields
    \begin{equation*}
        \norm{v(T) - u(T)}_{L^1(\R)} \leq \norm{v_0 - u_0}_{L^1(\R)} + 2 \eps |u_0|_{\BV(\R)}  + \frac{1}{\eps} \norm{A v - f(v)}_{L^1(\R \times (0, T))}
    \end{equation*}
    which implies the estimate \eqref{eq:main_stability_result} upon minimization with respect to $\eps$ (i.e.~setting $\eps = \sqrt{\norm{A v - f(v)}_{L^1} / 2 |u_0|_{\BV}}$).
\end{proof}

\begin{remark}
    Throughout the paper, the assumption that $f$ is globally Lipschitz is only used for simplicity of exposition. In fact, it suffices to require $f$ to be Lipschitz within the convex hull of $u_0(\mathbb{R})$. The constant $[f]_{\text{Lip}}$ can then be replaced by the Lipschitz constant of $f$ on this restricted interval.
\end{remark}

Until now, no assumptions on the initial distribution of particles $(x_0^i)_{i = 1}^N$ have been made. We specialize the result of Theorem~\ref{thm:main} to the case when there is a bound on the maximal distance $\dx_0^* \coloneqq \max_{i} \dx_0^i$ between initial particles.

\begin{proposition} \label{prop:explicit_convergence_rate}
    Let $u$ the entropy solution of \eqref{eq:conservation_law} and $v$ the approximation generated by the the particle path scheme. Assume that $f'$ is Lipschitz, and that there is a number ${R_{u_0} > 0}$ such that
    \begin{equation*}
        \int_{-\infty}^{x_0^1}  u_0(x)\, dx + \int_{x_0^N}^{\infty} u_0(x)\, dx \leq R_{u_0}.
    \end{equation*}
    Then the approximation error is bounded by
    \begin{equation*} \label{eq:general_stability_x}
        \begin{aligned}
            \norm{v(T) - u(T)}_{L^1(\R)} \leq R_{u_0} + |u_0|_{\BV(\R)} \biggl(\dx_0^* + 2\sqrt{T [f']_\lip \norm{u_0}_{L^\infty(\R)} \dx_0^*}\biggr).
        \end{aligned}
    \end{equation*}
\end{proposition}

\begin{proof}
    First, the standard estimate
    \begin{equation*}
        \begin{aligned}
            \int_{x_0^i}^{x_0^{i+1}}|u_0- v_0|\, dx & = \int_{x_0^i}^{x_0^{i+1}} \biggl|u_0(x) - \frac{1}{\dx_0^i} \int_{x_0^i}^{x_0^{i+1}} u_0(y)\, dy\biggr|\, dx \\
            & \leq \frac{1}{\dx_0^i} \int_{x_0^i}^{x_0^{i+1}}\int_{x_0^i}^{x_0^{i+1}} |u_0(x) - u_0(y)| \,dy\,dx \\
            & \leq \dx_0^i \norm{u_0}_{\BV([x_0^i, x_0^{i+1}])},
        \end{aligned}
    \end{equation*}
    where $|\cdot|_{\BV([x_0^i, x_0^{i+1}])}$ denotes the total variation on $[x_0^i, x_0^{i+1}]$, implies that
    \begin{equation*}
        \norm{v_0 - u_0}_{L^1(\R)} = R_{u_0} + \sum_{i = 1}^{N-1} \int_{x_0^i}^{x_0^{i+1}}|u_0- v_0|\, dx \leq R_{u_0} + \dx_0^* |u_0|_{\BV(\R)}.
    \end{equation*}
    It remains to estimate $\norm{A v - f(v)}_{L^1(\R \times (0, T))}$. Consider first $t \in (0, t_1)$, for which
    \begin{equation} \label{eq:interpolation_estimate}
        \int_\R |A(x, t) v(x, t) - f(v(x, t))|\, dx = \sum_{i = 1}^{N-1} v_t^i \int_{x_t^i}^{x_t^{i+1}} |A(x, t) - a(v_t^i)|\, dx.
    \end{equation}
    Since $A$ is defined as a linear interpolation, the supremum of the integrand can be bounded by
    \begin{equation*}
        \sup_{x \in (x_t^i, x_t^{i+1})}|A(x, t) - a(v_t^i)| \leq |V(v_t^{i-1}, v_t^i) - a(v_t^i)| + |V(v_t^{i-1}, v_t^i) - a(v_t^i)|.
    \end{equation*}
    As in the proof of Lemma~\ref{lemma:temporal_modulus}, let $\tilde{v}_t^i \in [v_t^{i-1}, v_t^i]$ be defined as the point at which $a$ attains the value $V(v_t^{i-1}, v_t^{i})$. Then
    \begin{equation*}
        \begin{aligned}
            \sup_{x \in (x_t^i, x_t^{i+1})}|A(x, t) - a(v_t^i)| & \leq |a(\tilde{v}_t^i) - a(v_t^i)| + |a(\tilde{v}_t^{i+1}) - a(v_t^i)| \\
            & \leq [a]_\lip \bigl(|\tilde{v}_t^i - v_t^i| + |\tilde{v}_t^{i+1} - v_t^i|\bigr).
        \end{aligned}
    \end{equation*}
    Plugging this into \eqref{eq:interpolation_estimate} gives
    \begin{equation*}
        \norm{A(t)v(t) - f(v(t))}_{L^1(\R)} \leq [a]_\lip \sum_{i = 1}^{N-1} v_t^i \dx_t^i \bigl(|\tilde{v}_t^i - v_t^i| + |\tilde{v}_t^{i+1} - v_t^i|\bigr).
    \end{equation*}
    Note that $[a]_\lip \leq [f']_\lip / 2$ and $v_t^i \dx_t^i = v_0^i \dx_0^i \leq v_0^* \dx_0^*$. This means that
    \begin{equation*}
        \begin{aligned}
            \norm{A(t)v(t) - f(v(t))}_{L^1(\R)} & \leq \frac{[f']_\lip}{2} v_0^* \dx_0^* \sum_{i = 1}^{N-1} \bigl(|\tilde{v}_t^i - v_t^i| + |\tilde{v}_t^{i+1} - v_t^i|\bigr) \\
            & \leq \frac{[f']_\lip}{2} v_0^* \dx_0^* |v_0|_{\BV(\R)} \\
            & \leq \frac{[f']_\lip}{2} \norm{u_0}_{L^\infty(\R)} \dx_0^* |u_0|_{\BV(\R)}
        \end{aligned}
    \end{equation*}
    for all $t \in (0, t_1)$. A similar argument can be made for each interval $(t_j, t_{j+1})$ between collisions. Inserting this in \eqref{eq:main_stability_result} finishes the proof.
\end{proof}

Corollary~\ref{cor:main} from the introduction follows immediately as a consequence  of Proposition \ref{prop:explicit_convergence_rate}.

\section{Conclusion and future work} \label{sec:discussion}

In this paper, we introduced and analyzed the particle path scheme \mbox{\ref{item:1}}--\mbox{\ref{item:3}}, a novel method for approximating entropy solutions of one-dimensional scalar conservation laws with non-negative initial data. The scheme is motivated by the equivalence between entropy solutions and well-posed particle paths, using a specific particle velocity derived from the Filippov particle velocity at discontinuities  of the solution.

The main contribution is the $L^1$-stability estimate presented in Theorem \ref{thm:main}, which bounds the error between the approximation $v$ and the true entropy solution $u$. This estimate directly leads to an explicit convergence rate of order $\mathcal{O}(\sqrt{\dx^*})$ under standard assumptions (Corollary \ref{cor:main}).

Furthermore, we demonstrated that for concave flux functions, the Particle Path Scheme coincides precisely with the well-known Follow-the-Leader (FtL) traffic flow model. This connection provides a new theoretical foundation for the FtL model and yields a rigorous proof of its convergence rate towards the macroscopic LWR model.

Throughout this work, we assumed non-negative initial data $u_0$, a standard assumption motivated by applications such as traffic modeling. To extend the scheme to signed initial data, one needs to handle the possible cancellation of mass where the solution changes sign. While the full particle velocity function \eqref{eq:velocityForRiemannParticle} derived in \cite[Theorem 1.6]{fjordholm_maehlen_oerke} applies to signed solutions, incorporating it into the scheme is technically more involved, and it is not immediately clear whether an analogous PDE formulation, similar to the continuity equation \eqref{eq:cont_eq} derived here, holds in this case.

A natural direction for future work is the development of fully discrete versions of the particle path scheme, obtained by applying numerical ODE solvers to the particle system \eqref{eq:particle_ode}. Although the stability and convergence of such schemes remains to be analyzed, a simple forward Euler discretization was employed to generate the numerical results shown in Figure \ref{fig:approximation_example}. Related work exists for the Follow-the-Leader model, where convergence of a fully discrete scheme was established in \cite{holden_risebro_2018}, albeit without an explicit convergence rate.

Finally, the connection between the scheme and the particle path interpretation of entropy solutions suggests an investigation of convergence of the characteristic flow. Specifically, one could study whether the flow map $X_t(s, x)$ generated by the ODE $\dot{x}_t = A(x_t, t)$ from Proposition~\ref{prop:representation} converges to the true Filippov flow associated with the entropy solution of the conservation law. Such convergence has interesting interpretations from a modelling perspective, and would provide further validation of the scheme's foundations.

\section*{Acknowledgements}
The author thanks Ulrik S.~Fjordholm for helpful comments on earlier drafts of this paper.

\bibliography{bibliography}
\bibliographystyle{abbrv}
\end{document}